\documentclass{article}
\usepackage{amssymb, amsthm, amsmath}
\usepackage{bm}
\usepackage{tikz}
\usepackage{xcolor}
\usetikzlibrary{calc}
\usetikzlibrary{decorations.pathmorphing,shapes}
\tikzstyle{v} = [draw, circle, fill=white, inner sep=2pt]
\usepackage{todonotes}

\usepackage{hyperref}
\usepackage[font=small,labelfont=bf, width=0.8\linewidth]{caption}

\usepackage{stmaryrd}
\theoremstyle{plain}
\newtheorem{thm}{Theorem}[section]
\newtheorem{theorem}[thm]{Theorem}
\newtheorem{lemma}[thm]{Lemma}

\newtheorem{proposition}[thm]{Proposition}

\newtheorem*{claim*}{Claim}
\newtheorem*{thm*}{Theorem}

\newcommand{\incs}{\subsetneq}
 \usepackage{mathtools}

\theoremstyle{definition}
\newtheorem{definition}[thm]{Definition}
\newtheorem{convention}[thm]{Convention}

\newtheorem{remark}[thm]{Remark}

\newtheorem{example}[thm]{\sc Example}

\newcommand{\inc}{\subseteq}
\newcommand{\union}{\cup}	
\newcommand{\Union}{\bigcup}	
\newcommand{\set}[1]{\{#1\}}
\newcommand{\setc}[2]{\set{#1 \mid #2}}
\newcommand{\inter}{\cap}
\newcommand{\hyper}[1]{{\bf #1}}
\newcommand{\restrconstr}[2]{{#1}_{{}^\lceil {#2}}}

\newcommand{\restrH}[2]{\hyper{#1}\backslash #2}

\newcommand{\bcdot}{\,{\bm \cdot}\,}

 \newcommand{\doublearrow}[1]{\xrightarrow[]{#1}\mathrel{\mkern-14mu}\rightarrow}
 \newcommand{\Friezo}{\hyper{F}}

\usepackage{float}
\usepackage[backend=bibtex]{biblatex}
  \title{Tridendriform algebras on hypergraph polytopes, the other way around}
  
  \usepackage{authblk}

\author[$^\dagger$]{  Pierre-Louis Curien}
\affil[$^\dagger$]{{\footnotesize Université Paris Cité,   CNRS, Inria, IRIF, Picube project-team,   Paris, France}}

\author[$^\ast$]{Bérénice Delcroix-Oger}
\affil[$^\ast$]{\footnotesize Université de Montpellier, CNRS, IMAG, Montpellier, France}

\author[$^\ddagger$]{Jovana Obradovi\' c } 
\affil[$^\ddagger$]{\footnotesize Mathematical Institute of the Serbian Academy of Sciences and Arts, Belgrade, Serbia}

\begin{document}
\maketitle

\begin{abstract}
\noindent Hypergraph polytopes (or nestohedra) form a broad class of polytopes obtained by truncating faces of a simplex according to a hypergraph. In earlier work, the authors constructed $q$-tridendriform algebras on the set of faces of certain families of hypergraph polytopes, including associahedra and permutohedra. The well-definedness of these   structures relied on a connectedness property on the hypergraphs involved, called strictness. Nevertheless, notable examples of hypergraph polytopes such as   cyclohedra fell outside this setting. We introduce a new connectedness condition, called anti-strictness, which goes opposite to strictness and captures a different class of hypergraph polytopes, including associahedra, permutohedra and cyclohedra. Our main result produces natural $(-1)$-tridendriform algebras in the  anti-strict framework, which match previously introduced tridendriform algebras in the overlap of the two frameworks, thereby extending the range of hypergraph polytopes admitting such algebraic structures. Our work is in fact formulated in an unbiased (or varying arity) version of the tridendriform structure, that we coined polydendriform in our previous work. {As a side contribution, we offer an operadic framework that crystallises this notion.}
\end{abstract}

\section*{Introduction}
In the previous work \cite{PLBJ1}, 
we introduced $q$-tridendriform algebras (cf. \cite{BR}) on the sets of faces of certain families of hypergraph polytopes (a.k.a. nestohedra \cite{P09}), encompassing the classical constructions  on associahedra and permutohedra due to Loday-Ronco and Palacios-Ronco (namely, shuffles of trees and shuffles of surjections). Hypergraph polytopes are the   generalisation of graph associahedra which allows (hyper)edges of arbitrary cardinality, not only of cardinality 2. 
Just like graph associahedra, hypergraph polytopes admit a purely combinatorial description of their face posets. It is this description that underlies the  algebraic structures that we study. 

\smallskip

 The general construction is the following.  We consider a finite family of connected hypergraphs $\{\hyper{H}_a \,|\, a\in A\}$ with pairwise disjoint vertex sets $H_a$, together with a connected hypergraph $\hyper{H}$ whose vertex set $H$ is the disjoint union of the sets $H_a$. These data form  what we call a {\it preteam}. A preteam 
serves as the arity 
of an operation that   takes as input an $A$-indexed  family of faces $C_a$, where  each $C_a$ is a face of the hypergraph polytope associated with $\hyper{H}_a$,  and returns a face of the hypergraph polytope associated with  $\hyper{H}$.  Conceptually, this operation is   a kind of generalised shuffle:  the input faces may be thought of as   structured ``sets of cards" that are interweaved according to the combinatorics of $\hyper{H}$. 

\smallskip

 For our general theory to work, we had to fix some condition on preteams. This condition, that we called {\it strictness}, requires that whenever $K$ is a connected subset of $H_a$ (according to  the hypergraph structure $\hyper{H}_a$), then $K$ remains  connected when viewed inside $\hyper{H}$.  We have also designed
a slight weakening of the strict framework, in terms of the  {\it quasi-strictness} condition, allowing us to encompass simplices and hypercubes. 

\smallskip

However, these frameworks did not cover certain natural and important families of polytopes. In particular, cyclohedra \cite{Markl-Cyc,Simion} could not be accommodated, despite the fact that an 
explicit tridendriform algebra structure could be constructed for them by hand. In our search for yet another ``variant" of our condition that would encompass the case of cyclohedra,
we discovered instead a different general and {\it{incomparable}} framework. Interestingly, the new framework reverses the direction of connectivity preservation: it requires that, whenever $K$ is a connected subset of $H$ (according to the hypergraph structure $\hyper{H}$), then $K\cap H_a$ is also connected in $\hyper{H}_a$, for all $a$ such that $K\cap H_a$ is non-empty. 
We call this condition {\it{anti-strictness}}. The resulting framework allows us to treat cyclohedra, as well as other examples, such as associahedra assembled by insertion rather than by concatenation of the underlying linear graphs.  It also turns out that many examples previously studied in \cite{PLBJ1} still fit in the anti-strict setting as well, in the sense that they are both strict and anti-strict.  Moreover, the proof that we get a polydendriform structure (which is the varying-arity version of tridendriform structure) is nicer, and even the definition of the shuffle product is more elegant, and coincides with the one in  \cite{PLBJ1} for families happening to be both strict and anti-strict.

\smallskip

This paper reports on these new findings. It can be read mostly independently from our previous work, but we refer to it (and to \cite{COI}, as well as to the abundant literature on graph associahedra) for a more leisurely introduction to hypergraph polytopes and their geometric realisation as truncations of   simplices.

\smallskip

The paper is organised as follows. In Section \ref{reminders-hypergraphs-section}, we recall basic facts about hypergraph polytopes and their combinatorial description in terms of  tree-like  objects called constructs. We also review and discuss the notion of restriction of a construct,  which  plays an essential role in the sequel. In Section \ref{anti-strict-section}, we define  anti-strict  teams and clans, and 
give examples that in particular reflect the incomparability of strict and anti-strict frameworks. In Section \ref{shuffle-section}, we introduce a  shuffle product for anti-strict clans, and  prove our main associativity result. {We also place it in an operadic context.}
  Section \ref{comparison-section} compares our framework with related constructions in the literature, in particular those of Ronco, Chapoton, Forcey and Springfield.  We conclude with an appendix presenting a non-inductive characterisation of restriction of a construct, which may be of independent interest.

\section{Hypergraph polytopes}
\label{reminders-hypergraphs-section}
In this section, we recall the basic combinatorial framework of hypergraph polytopes, fixing notation and terminology. Sections \ref{hypergraphs-subsection}, \ref{constructs-subsection}   and 
\ref{restriction-subsection} contain the definitions of hypergraphs, their constructs and  restrictions thereof (defined inductively), respectively, while Section \ref{fundamental-examples-section} provides fundamental examples of hypergraph polytopes. 
Most of the  material in these four sections   follows closely the earlier work  \cite{PLBJ1} and is recalled here for completeness. Section \ref{restriction-characterisation-subsection}  introduces a non-inductive characterisation of restriction, which is new.

\subsection{Hypergraphs} \label{hypergraphs-subsection}
 
A hypergraph   consists of a  set  $H$ of {\em vertices}  and a subset 
$\hyper{H}\inc {\mathcal{P}}(H)\backslash\emptyset$ such that $\Union \hyper{H}=H$, whose elements are called  \emph{hyperedges}.  In addition, we   assume that 
 $\set{x}\in \hyper{H}$, for all $x\in H$ (in terminology of \cite{DP} and \cite{PLBJ1}, our hypergraphs are {\it{atomic}}). Throughout the paper, we shall use the same bold symbol $\hyper{H}$ to denote both an  entire hypergraph  and its set  of hyperedges. A hyperedge of cardinality 2 is called an \emph{edge}.  Note that any ordinary graph $(V,E)$ can be viewed as the hypergraph
$\setc{\set{v}}{v\in V} \union \setc{e}{e\in E}$ (with no hyperedge  of cardinality $\geq 3$).

\smallskip
 
Given a hypergraph $\hyper{H}$    and  a subset $X\inc H$, we define the restriction of $\hyper{H}$ to $X$ as the hypergraph
$\hyper{H}_X:=\setc{Z}{Z\in \hyper{H}\;\mbox{and}\; Z\inc X}$, and we introduce the abbreviation $\restrH{H}{X}:=\hyper{H}_{H\backslash X}$. We say that $\hyper{H}$ is \emph{connected} if there is no non-trivial partition $H=X_1\union X_2$ such that $\hyper{H}=\hyper{H}_{X_1}\union \hyper{H}_{X_2}$, and that $X\inc H$ is connected in $\hyper{H}$
if $\hyper{H}_X$ is connected.
For each finite hypergraph there exists a unique partition
$H=X_1\union\ldots\union X_m$ such that each $\hyper{H}_{X_i}$ is connected and $\hyper{H}=\Union(\hyper{H}_{X_i})$.  The $\hyper{H}_{X_i}$ are called  the \emph{connected components} of $\hyper{H}$. If $\hyper{H}_1,\ldots,\hyper{H}_n$ are  the
 connected components of $\restrH{H}{X}$, we write
$\hyper{H},X  \leadsto \hyper{H}_1,\ldots, \hyper{H}_n$.
 
 \begin{convention} \label{restriction-convention}
 If $H$ and $Z$ are subsets of a common subset and $H\cap Z\neq\emptyset$, we shall often simply write
 $\hyper{H}_Z$ for $\hyper{H}_{Z\cap H}$.
 \end{convention}

\subsection{Constructs} \label{constructs-subsection}
As shown by  Do\v sen and Petri\'c~\cite{DP},  the non-trivial connected subsets of a finite connected hypergraph ${\bf H}$ 
encode the truncation data for a simplex with  $|H|$ vertices. 
The faces of the resulting polytope  admit a combinatorial description in terms of certain
non-planar trees, called \emph{constructs}, whose nodes are decorated by non-empty subsets of $H$. We now recall their recursive definition, using the syntax introduced in \cite{COI}.

\smallskip

Let  $\emptyset\neq Y\subseteq H$. If   $\hyper{H},Y  \leadsto \hyper{H}_1,\ldots, \hyper{H}_n$, and if  $T_1,\ldots,T_n$ are constructs of $\hyper{H}_1,\ldots,\hyper{H}_n$, respectively, then the tree obtained by grafting $T_1,\ldots,T_n$ on the root node decorated by $Y$, denoted by $Y(T_1,\ldots,T_n)$ (or sometimes $Y\setc{T_i}{1\leq i\leq n}$), is a construct of  $\hyper{H}$. We write $Y=\mbox{root}(Y(T_1,\ldots,T_n))$. 
 The base case is when $Y=H$ (and hence $n=0$): then the one-node tree $H()$ (written simply $H$) is a construct.
We write $T:\hyper{H}$ to denote that $T$ is a construct of $\hyper{H}$.

\begin{convention}
We   identify the vertices of constructs with the sets decorating them. In addition, to simplify notation, we omit braces around singleton vertices. For example, instead of 
$\{x\}(\{u,v\},\{y\})$   and  $\{x\}(\{u\}(\{v\}),\{y\})$, we shall write $x(\{u,v\},y)$ and $x(u(v),y)$, respectively.
\end{convention}

In the geometric realisation, the dimension of the face encoded by a construct $T$ is  $\sum_X (|X|-1)$, where
$X$ ranges over all nodes of $T$. In particular, the vertices of the polytope realising $\hyper{H}$ are the constructs of $\hyper{H}$ whose nodes are all singletons. They are called \emph{constructions}.

\smallskip

Constructs can be ordered by edge contractions. More precisely, given a construct $S$ with two adjacent nodes 
$X,Y$, we can contract the edge connecting them into a single node labelled $X\cup Y$. This captures exactly the immediate subface relation in the geometric realisation. 

\smallskip

The polytopes realised in this way are  called \emph{hypergraph polytopes}, after Do\v sen and Petri\' c \cite{DP}, or \emph{nestohedra},  after Postnikov \cite{P09}.  
In the special case where $\hyper{H}$ is a graph,  we recover the extensively investigated notion of graph associahedra. 

\subsection{Fundamental examples of hypergraph polytopes} \label{fundamental-examples-section}

\smallskip
The polytopes for the complete graphs
$$\hyper{P}^X= \set{\set{x_1},\ldots,\set{x_n},\set{x_1,\ldots x_n}} \cup \setc{\set{x_i,x_j}}{1\leq i\neq j\leq n},$$
with  $X=\set{x_1,\ldots,x_n}$, are the $(n-1)$-dimensional \emph{permutohedra}. Since each non-empty subset of $X$ is connected in $\hyper{P}^X$, the constructs of $\hyper{P}^X$ correspond to ordered partitions of $X$, or, equivalently,  to surjections $X\doublearrow{}\{1,\dots,k\}$.

\smallskip
For a totally ordered set $X=\{x_1<\dots < x_n\}$, the polytope for the hypergraph 
$${\bf K}^{X}=\{\{x_1\},\dots,\{x_n\},\{x_1,x_2\},\dots,\{x_{n-1},x_n\}\}$$
is called the $(n-1)$-dimensional \emph{associahedron}. The constructs of ${\bf K}^{X}$ are in one-to-one correspondence with  partially parenthesised words, or, equivalently, with  Schr\" oder trees with at least two leaves  (the equivalence of these three characterisations of the faces of the associahedra is illustrated in \cite[Section 2.4]{COI}).

\smallskip

For a finite cyclically ordered set $X=\set{x_1<\ldots<x_n<x_1}$,  the polytope for the hypergraph
$${\bf O}^X=\set{\set{x_1},\ldots,\set{x_n},\set{x_1,x_2},\ldots,\set{x_{n-1},x_n},\set{x_n,x_1}}$$ is called the $(n-1)$-dimensional \emph{cyclohedron}.
The constructs of ${\bf O}^X$ correspond to   partially parenthesised cyclic words, through their equivalence with tubings or nested sets (see \cite[Proposition 2]{COI} and \cite[Lemma 1.4]{Markl-Cyc}).
 
\smallskip 
All examples so far are graph associahedra. There are two fundamental examples of genuine hypergraph polytopes, with which we close our list for now (later, we shall encounter erosohedra, friezohedra, stellohedra and teleassociahedra).

\smallskip

For  a finite set $X$, the $(|X|-1)$-dimensional   \emph{simplex} is encoded by the hypergraph   $$\hyper{S}^X=\setc{\set{x}}{x\in X}\union\set{\set{X}},$$ having only trivial hyperedges (i.e., the vertices and the maximal hyperedge ensuring that the hypergraph is connected). The constructs have the form $Y\{x_1,\dots,x_k\}$, where $\emptyset\incs Y\inc X$ and $\{x_1, \ldots, x_k\}=X\backslash Y$.

\smallskip
For a finite ordered set $X=\set{x_1<\cdots<x_n}$, the polytope for the hypergraph 
 $$\hyper{C}^X=\{\{x_1\},\dots,\{x_n\}\}\cup\setc{\setc{x_j}{1\leq j\leq i}}{1\leq i\leq n}.$$
is the $(n-1)$-dimensional \emph{hypercube}. The constructs of $\hyper{C}^X$ are in one-to-one correspondence with the set of words of
length $n$ over the alphabet $\{+,-,\bullet\}$, starting with $+$ (see \cite[Example 3.7]{PLBJ1} for details).

\subsection{Restriction} \label{restriction-subsection}
We recall below the definition of the restriction $\restrconstr{C}{K}$ of a construct  $C:{\bf H}$ to a non-empty connected subset $K$ of $H$.
 \begin{definition}[Restriction of a construct] \label{restriction-definition}
For a hypergraph ${\bf H}$, a non-empty subset $K\subseteq H$ such that the hypergraph ${\bf K}:={\bf H}_K$ is connected, and a construct $C:{\bf H}$,   \emph{the restriction of $C$ to $K$} is  the construct $C_{\lceil K}:{\bf K}$ defined as follows:
\begin{itemize}
\item if $C=H$, then $C_{\lceil K}=K$, 
\item if $C=X(C_1,\dots,C_n)$, where $\emptyset \neq X \subsetneq H$, ${\bf H},X\leadsto {\bf H}_1,\dots ,{\bf H}_n$ and $C_{i}:{\bf H}_{i}$ for $1\leq i\leq n$, and  
\begin{itemize}
\item if $X\inter K=\emptyset$, then there exists  $i\in\{1,\dots,n\}$ such that ${\bf K}$ is connected in ${\bf H}_i$, and we set   $C_{\lceil K}={(C_{i})}_{\lceil_{K}}$;
\item otherwise,  supposing that ${\bf K},(X\inter K)\leadsto K_1,\dots,K_p$, we have that  for each $j\in\{1,\dots,p\}$  there exists  $i\in\{1,\dots,n\}$, such that $K_j$ is connected in $H_i$; we denote by $\psi:\{1,\dots,p\}\rightarrow \{1,\dots,n\}$  the induced index correspondence, and we set $$C_{\lceil_{K}}=(X\inter K)((C_{\psi(1)})_{\lceil_{K_1}},\dots,(C_{\psi(p)})_{\lceil_{K_p}}).$$
\end{itemize}
\end{itemize}
\end{definition}
A non-inductive characterization of the notion of restriction   of a construct  can be found in the Appendix. 
\begin{remark}\label{rest1}
For connected subsets $\emptyset \neq L\subseteq K$ and $\emptyset\neq K\subseteq H$  and a construct $C:{\bf H}$,  it holds that ${(C_{\lceil_K})}_{\lceil_L}=C_{\lceil_L}$.
\end{remark}

In the following example, we illustrate a situation in which the restriction of a construct ``breaks'' the ancestry relation.  
\begin{example}
For the hypergraph $${\hyper{H}}=\{\{1\},\{2\},\{3\},\{4\},\{1,2\},\{2,3\},\{3,4\},\{1,4\}\},$$ the construct $U=1(2(3(4)))$ of ${\bf H}$ and the connected subset $K=\{1,2,4\}$ of ${\bf H}$, we have that $\restrconstr{U}{K}=1(2,4)$. Therefore, the node  $2$ used to be strictly below the node  $4$ in $U$, while, in $\restrconstr{U}{K}$, these two nodes became incomparable. 
\end{example}
\section{Anti-strict teams}\label{anti-strict-section} 

In this section, we introduce the main notion of the paper: {\em anti-strict teams}. We also introduce the corresponding notion of anti-strict clans and present examples illustrating the incomparability   between strict and anti-strict frameworks. 

\smallskip

We begin by recalling from~\cite{PLBJ1}   that a preteam $\tau$ consists of
 a finite collection   $\setc{\hyper{H}_a}{a\in A}$ of connected hypergraphs with disjoint sets of vertices,  together with
a  connected hypergraph $\hyper{H}$, such that $H=\Union_{a\in A}H_a$.
We use the notation $\tau=(\setc{\hyper{H}_a}{a\in A},\hyper{H})$ and we refer to the $\hyper{H}_a$'s as the participating hypergraphs of $\tau$ and to $\hyper{H}$ as the coordinating hypergraph  of $\tau$.

\begin{definition}[Anti-strict team] \label{anti-strict-definition}
A preteam $\tau=(\{{\bf H}_a\,|\,a\in A\},{\bf H})$ is called an \emph{anti-strict team} if it satisfies the following property:
\begin{itemize}
\item[(AS)] for every subset  $\emptyset\neq K\inc H$, if $K$ is connected in $\hyper{H}$, then for all $a\in A$ such that $H_a\inter K\neq\emptyset$,   $H_a\inter K$ is connected in $\hyper{H}_a$.
\end{itemize} 
\end{definition}

The notion of strictness and anti-strictness can be illustrated in a cobordism-like manner. We do this in Figure \ref{figInt}.

\begin{figure}[H] 
\centering
\begin{tabular}{ccc}
\includegraphics[height=2.5cm]{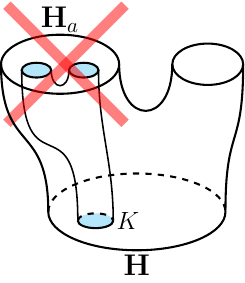} & \hspace{2cm} &
\includegraphics[height=2.5cm]{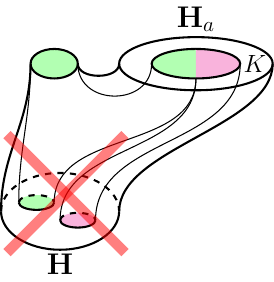} \\
(AS) & & (S)
\end{tabular}
\caption{A preteam can be represented as a cobordism whose upper and lower boundary disks
feature the participating and coordinating hypergraphs, respectively. The forbidden configuration in the anti-strict case  (AS)  is given by a connected subset $K$ in the lower disk which disconnects into more than one connected component in some participating hypergraph ${\bf H}_a$.   The forbidden configuration in the strict case (S) is given by a    connected subset $K$ in one of the upper disks which disconnects into more than one connected component in the coordinating hypergraph {\bf H}.}\label{figInt}
\end{figure}

The following lemma provides an alternative characterisation of anti-strict team.
Let $\tau=(\{{\bf H}_a\,|\,a\in A\},{\bf H})$ be a preteam. Pick  a   subset $\emptyset\neq B\subseteq A$ and  for each $b\in B$  a subset $\emptyset\neq X_b\subseteq H_b$, and
consider the decompositions
\begin{equation}\begin{array}{l}\hyper{H}_b, X_b \leadsto \hyper{H}_{(b,1)},\ldots ,\hyper{H}_{(b,n_b)}\\
\hyper{H},X_B\leadsto \hyper{H}_1^B,\ldots,\hyper{H}_{n_B}^B\quad(\textrm{where}\; X_B=\Union_{b\in B}X_b).
\end{array}\label{decomp}\end{equation}

\begin{lemma} \label{daggerbis}
A preteam $\tau=(\{{\bf H}_a\,|\,a\in A\},{\bf H})$ is an anti-strict team if and only if the following property holds: for each choice of  subsets $\emptyset\neq B\subseteq A$ and   $\emptyset\neq X_b\subseteq H_b$  for which $n_B=1$ in the decomposition \eqref{decomp}, we have that, for each $b\in B$, either $n_b=0$ or $n_b=1$.
\end{lemma}

\begin{proof}
    The statement of the lemma  is a matter of   mere  repackaging of Definition \ref{anti-strict-definition}. More precisely, it follows   from the following sequence of observations:
\begin{itemize}  
\item The data of $B$ and of the sets  $X_b$ amount to the data of some non-empty $X\inc H$ (and then $B=\setc{b\in A}{H_b\inter X\neq\emptyset}$). 
The assumption $n_B=1$ amounts to say that $X=H\setminus K$ for a connected subset $K$ in $\hyper{H}$. 
\item The assumption $n_b=0$ amounts to say that $H_b \inter K$ is empty, while the assumption $n_b=1$ says that
$H_b \inter K$ is connected in $\hyper{H}_b$.
\end{itemize}
 \end{proof}

  We next define the notion of universe and of anti-strict clan. Informally, it is a collection of anti-strict teams satisfying a closure condition that  will ensure the correctness of our definition of shuffle product in the next section.
To proceed, we need a notation.
For an anti-strict team $\tau=(\{{\bf H}_a\,|\, a\in A\},{\bf H})$ and a connected subset $\emptyset \neq K\subseteq H$,  we define an induced preteam $\tau_K$ as follows  (using Convention \ref{restriction-convention}):
\begin{equation}\tau_K=(\{{(\hyper{H}_{{a}})_{K}\,|\,  a \in  A\;\mbox{and}\;H_{{a}}\inter K\neq\emptyset\},{\bf K}}).\label{lll}\end{equation}

\begin{definition}  \label{antistrict-clan-definition}
A \emph{universe} is a collection $\mathfrak{U}$ of connected hypergraphs. An \emph{anti-strict clan} on some universe $\mathfrak{U}$ is a set $\Xi$ of anti-strict teams such that, for each $\tau\in\Xi$, the participating hypergraphs and the  coordinating hypergraph of  $\tau$  belong to $\mathfrak{U}$, and which is closed under decomposition, in the sense that,   $\tau_K\in \Xi$ for each  connected subset  $\emptyset\neq K\subseteq H$.
\end{definition}

We now present two fundamental examples of anti-strict clans. 

\begin{example}[The clan of trees of associahedra]\label{treesassoc}
As universe, we fix the class of  all associahedra. We are going to describe preteams in terms of planar trees with leaves, as follows. Let $T$ be a rooted planar tree with  set of nodes $A$. For $a\in A$, denote by $l(T,a)$ the   set of leaves of $T$ attached to $a$, endowed with the total order induced by planarity. Our preteams have the form $$(\{{\bf K}^{l(T,a)}\,|\,a\in A\}, {\bf K}^{l(T)}),$$ where $l(T)$ is the set of all leaves of $T$, again ordered by planarity. In order to see that such a preteam is an anti-strict team, we note that each connected subset $\emptyset\neq K\subseteq {K}^{l(T)}$ is a  subinterval of $l(T)$, which entails, by planarity, that for each $a\in A$, $l(T,a)\cap K$ is an interval on $l(T,a)$ (or empty), and hence connected in ${\bf K}^{l(T,a)}$. We define the clan of trees of associahedra as the set of all teams defined as above. It remains to verify that this set of teams is closed under restriction to connected subsets.   Indeed, each connected $\emptyset \neq K\subseteq K^{l(T)}$ is an interval  on  $l(T)$, which in turn  determines a subtree $T'$ of $T$ having precisely $K$ as its set of leaves. It is easily seen that  $\tau_K$ is determined precisely by $T'$, and hence belongs to the clan. 

We note that this clan is not strict. Indeed, consider  the tree $T$ with two nodes $a$ and $b$, with  root $a$ having as inputs   leaf 1, the edge from $b$ and  leaf 4 (in this order), and  $b$ has no incoming internal edge and has 2,3 (in this order) as leaves. We can write $T$  in term form as $a(1,b(2,3),4)$.

 Then in the corresponding preteam we have that $\set{1,4}$ is connected in ${\bf K}^{l(T,a)}={\bf K}^{\set{1,4}}$ but is not connected in ${\bf K}^{l(T)}={\bf K}^{\set{1,2,3,4}}$.
\end{example}
 
\begin{example}[The clan of trees of associahedra and cyclohedra]
\label{excyc} 
Recall from Section \ref{fundamental-examples-section} the definition of cyclohedra.
Notice that, for each connected subset $\emptyset \subsetneq H\subsetneq O^X$, we have that $({\bf O}^X)_H={\bf K}^H$, where $H$ is endowed with the total order induced from $X$. Our universe will be formed by  all associahedra and  cyclohedra. We shall describe preteams of two different forms: one in terms of planar rooted trees, for which the coordinating hypergraph is an associahedron, and the other one in terms of planar non-rooted trees,  for which the coordinating hypergraph is a cyclohedron. 
\begin{itemize}
\item[(A)] Let, like in Example \ref{treesassoc},  $T$ be a rooted planar tree with  set of nodes $A$ and with the set of leaves $l(T)$, ordered by planarity. In addition, denote by $l(T,a)$  the   set of leaves of $T$ attached to $a\in A$. Our preteams will have the form   $$(\{{\bf H}^{l(T,a)}\,|\,a\in A\}, {\bf K}^{l(T)}),$$ where each ${\bf H}^{l(T,a)}$ is either ${\bf K}^{l(T,a)}$, with ${l(T,a)}$  endowed with the total order induced by planarity, or  ${\bf O}^{l(T,a)}$, in which case $l(T,a)$  is cyclically ordered by extending in the obvious way the total order given by planarity.

\item[(B)] In order to describe the preteams whose coordinating hypergraph is a cyclohedron, we start with an unrooted planar tree $U$ with the set of nodes $A$ and the set of leaves $l(A)$, cyclically ordered by planarity. The set $l(U,a)$ of leaves  of $U$ attached to $a\in A$ comes equipped with the induced cyclic order. The corresponding preteams now have the form $$(\{{\bf O}^{l(U,a)}\,|\,a\in A\}, {\bf O}^{l(T)}).$$  
\end{itemize}
In order to show that  preteams of both forms are indeed   anti-strict teams, we use exactly the same argument as in  Example \ref{treesassoc}. We define the clan of trees of associahedra and cyclohedra as the set of all these teams.    Our goal now is to show that both forms of teams are closed under restrictions to connected subsets.

 For teams of form (A), as in Example \ref{treesassoc}, the choice of a connected subset $\emptyset\neq K\subseteq K^{l(T)}$ determines the subtree $T_K$ of $T$ such that $l(T_K)=K$. Let us define $$\tau_{T_K}:=(\{{\bf H}^{l(T_K,a)}\,|\, a\in \mbox{nodes}(T_K)\},{\bf K}),$$ where
$${\bf H}^{l(T_K,a)}=\begin{cases}
{\bf O}^{l(T_K,a)},  & \mbox{if } {\bf H}^{l(T,a)}={\bf O}^{l(T,a)} \mbox{ and }   l(T_K,a)= l(T,a) \\
{\bf K}^{l(T_K,a)}, & \mbox{otherwise}.
\end{cases}$$
This definition is motivated as follows. If ${\bf H}^{l(T,a)}={\bf O}^{l(T,a)}$ and $l(T_K,a)\subsetneq l(T,a)$, then ${\bf O}^{l(T,a)}\lceil_{ l(T_K,a)}={\bf K}^{l(T_K,a)}$. Then it is easy to verify that $\tau_{T_K}$ is precisely $\tau_K$, and hence, $\tau_K\in \Xi$.

 For teams of form (B),    the choice of a connected subset $\emptyset\neq K=\{k_1<\dots<k_p\}\subseteq O^{l(T)}$ determines the subtree $U_K$ of $U$ such that $l(U_K)=K$. Let $T_{K}$ be the rooted tree obtained from $U_K$ by rooting it at the node $a\in A$ for which $k_1\in l(U_K,a)$, in such a way that the  planar order of the leaves of $T_K$ reads precisely $k_1<\cdots <k_p$. We define $\tau_{T_K}$ in exactly the same way as above and conclude by the same argument.
 
 We conclude by noting that this example is not strict either since it contains the clan of Example~\ref{treesassoc}.
\end{example}

 We shall give more examples of anti-strict clans, as well as two examples of strict but not anti-strict clans, in Section \ref{strict-and-antistrict-coincidence-section}.

\section{Polydendriform shuffle product} \label{shuffle-section}
In Section \ref{shuffle-definition-section}, we  define our shuffle product for anti-strict clans and provide examples. Our definition uses a fixed real parameter $q$. In Section \ref{associativity-section}, we prove that,  by specialising to $q=-1$,  our product becomes  associative. More precisely, this associative structure is decomposed into a finer structure, called polydendriform.
In Section \ref{strict-and-antistrict-coincidence-section}, we show that the present shuffle product coincides with the one defined in \cite{PLBJ1} for clans that are both strict and anti-strict. {In Section \ref{operadic-section}, we place our work  in an operadic context.}

\subsection{Definition and examples}
 \label{shuffle-definition-section}
Let $\tau:=(\setc{\hyper{H}_a}{a\in A},\hyper{H})$ be a preteam. Recall from \cite{PLBJ1} that a {\it delegation of support} $\tau$ is a family of constructs $\delta=\{C_a:{\bf H}_a\,|\, a\in A\}$, one for each participating hypergraph. We write $\delta = (\{C_a:{\bf H}_a\,|\, a\in A\},{\bf H})$ to take note of the coordinating hypergraph ${\bf H}$ as well.
We first define, for each connected subset $\emptyset\neq K\subseteq H$, the restriction of $\delta$ to $K$ as the delegation $$\delta_{\lceil_{K}}\, :=\, (\{{C_a}_{\lceil_{H_a\cap K}}\,|\, a\in A,\, H_a\cap K\neq \emptyset\},{\bf K})$$ of $\tau_K$. Suppose that
$\mbox{root}(C_a)=X_a$, for all $a\in A$. For a non-empty subset $\emptyset\neq B\subseteq A$, we set $X_B:=\bigcup_{b\in B}X_b$ and, supposing that ${\bf H},X_B\leadsto {\bf H}_1,\dots,{\bf H}_n$, we define recursively
$$\ast_B(\delta)=X_B(\ast(\delta_{\lceil_{H_1}}),\ldots,\ast(\delta_{\lceil_{H_n}})).$$ Finally, the  {\it   shuffle product}  (or   product) of   $\delta$ is  defined by
\begin{equation} \label{shuffle-antistrict-def} 
\ast(\delta)= \sum_{\emptyset\neq B\subseteq A} q^{{|B|}-1}\ast_B(\delta),
\end{equation}
where $q$ is a fixed real number.  Note that $\ast(\delta)$ is a linear combination of constructs, living in the vector space spanned by the constructs of $\hyper{H}$. We shall  write $\ast^{\tau}(\delta)$ instead of $\ast(\delta)$ whenever we wish to record explicitly the underlying support team 
$\tau$.

\smallskip

Forgetting about the coefficients (or setting $q=1$), this definition has the following algorithmic (and ludic) reading: a summand of $\ast(\delta)$ is obtained by  choosing a subset $B$ of the set $A$ of ``trees of sets of cards", and by picking simultaneously the set of cards  at the root (or ``top'') of each chosen tree, and then continuing this game recursively.  Note that if at each stage one chooses $B$ to be a singleton, and if each $C_a$ is a construction, then we have really the ``feeling'' of shuffling trees of cards in a sense closer to the one familiar to card game players.

\smallskip
In the following two examples, we illustrate  the shuffle product on the clans of Example \ref{treesassoc} and Example \ref{excyc}, respectively. 
\begin{example} Let us take ${\bf H}={\bf K}^{\set{1<2<1'<2'<3}}$ 
 and let $T$ (in term form) be the tree $a_1(1,2,a_2(1',2'),3)$, and define $\tau=(\set{{\bf H}_{a_1}, {\bf H}_{a_2}},{\bf H})$, where
 ${\bf H}_{a_1}={\bf K}^{\set{1,2,3}}$ and ${\bf H}_{a_2}={\bf K}^{\set{1',2'}}$. 
 Consider now the delegation$(\set{C_{a_1},C_{a_2}},\hyper{H})$, where 
$C_{a_1}=\set{1}(\set{2,3})$ and $C_{a_2}= \set{1'}(\set{2'})$. We compute, say
$$\begin{array}{lll}
\ast_{\set{a_2}}(C_{a_1},C_{a_2}) & = &\set{1'}(\restrconstr{(C_{a_1})}{\set{1,2}},\ast(\restrconstr{(C_{a_2})}{\set{2'}},\restrconstr{(C_{a_1})}{\set{3}})\\
& = & \set{1'}(\set{1}(\set{2}),\ast(\set{2'},\set{3}))\\
& = & \set{1'}(\set{1}(\set{2}), \set{2'}(\set{3})+\set{3}(\set{2'})-\{2',3\}).
\end{array}$$

\end{example}
\begin{example} Our ambient anti-strict team will be
$$\tau=(({\bf O}_{\{1,\dots,m\}},{\bf O}_{\{m+1,\dots,n\}}),{\bf O}_{\{1,\dots,m,m+1,\dots,n\}}).$$ As a delegation, we take $C_1=\{i\}(C_{1,1}):{\bf O}_{\{1,\dots,m\}}$ and $C_2=\{j\}(C_{2,1}):{\bf O}_{\{m+1,\dots,n\}}$, where $i\notin\{1,m\}$ and $j\not\in\{m+1,n\}$. We then have
$$\begin{array}{rcl}
{\bf O}_{\{1,\dots,n\}}\backslash\{i,j\} &\leadsto &   {\bf K}_{\{i+1,\dots,j-1\}},  {\bf K}_{\{j+1,\dots,n,1,\dots,i-1\}},
\end{array}$$

$$\begin{array}{rcl}
\tau_1&=&(({\bf K}_{\{i+1,\dots,m\}},{\bf K}_{\{m+1,\dots,j-1\}}),{\bf K}_{\{i+1,\dots,j-1\}}),  \\[0.1cm]
\tau_2&=&(({\bf K}_{\{j+1,\dots,n\}},  {\bf K}_{\{1,\dots,i-1\}}),{\bf K}_{\{j+1,\dots,n,1,\dots,i-1\}}),\end{array}$$ 

and

$$\begin{array}{rcl}
\delta_1&=((({C_1})_{\lceil_{\{i+1,\dots,m\}}},({C_2})_{\lceil_{\{m+1,\dots,j-1\}}}),{\bf K}_{\{i+1,\dots,j-1\}}),&\\[0.1cm]
\delta_2&=((({C_2})_{\lceil_{\{j+1,\dots,n\}}}, ({C_{1}})_{\lceil_{\{1,\dots,i-1\}}}),{\bf K}_{\{j+1,\dots,n,1,\dots,i-1\}}),&\end{array}$$ 
and we calculate
$$\ast_{\{1,2\}}(C_1,C_2)=\{i,j\}(\ast(\delta_1),\ast(\delta_2)).$$
 
\end{example}

\subsection{Polydendriform structure} \label{associativity-section} \label{restriction-characterisation-subsection}
We first prove a  useful property of commutation of the shuffle product  \eqref{shuffle-antistrict-def} with restriction. With this goal, we extend the  notion of restriction from constructs to linear constructs in the obvious way. In addition, from now on, we suppose $q=-1$.
 
\begin{proposition} \label{shuffle-restriction}
Let $\delta = (\{C_a:{\bf H}_a\,|\, a\in A\},{\bf H})$  be a delegation and let $\emptyset\neq K\subseteq H$ be a connected subset of $H$. For a subset $B\inc A$, set  $B_1:=\setc{b\in B}{X_b\inter K\neq\emptyset}$ and
$B_2:=\setc{b\in B}{X_b\inter K=\emptyset}$. Then the following properties hold.
\begin{itemize}
\item[{\em (P$_1$)}] {\it If $B_1\neq\emptyset$, then $\restrconstr{(\ast_B(\delta))}{K}=\restrconstr{(\ast_{B_1}(\delta))}{K}$.}
\item[{\em (P$_2$)}] {\it If $B_1=\emptyset$, then $\restrconstr{(\ast_B(\delta))}{K}=\ast(\restrconstr{\delta}{K})$.}
\item[{\em (P)}] The shuffle product is compatible with restriction, i.e., $\restrconstr{(\ast(\delta))}{K}=\ast(\restrconstr{\delta}{K})$.
\end{itemize}
\end{proposition}
\begin{proof} We shall prove   the three  properties   together, by induction on the size of ${\bf H}$.  

\smallskip

 We   prove  (P$_1$) first. We set $X_B=\Union_{b\in B}X_b$ and $X_{B_1}=\Union_{b\in B_1}X_b$, inducing
$$\begin{array}{rcl}

\hyper{H},X_B & \leadsto & \hyper{H}_1,\ldots,\hyper{H}_n\\ 
\hyper{H},X_{B_1} & \leadsto & \hyper{H}^1_1,\ldots,\hyper{H}^1_m \\
\hyper{H}_K,X_{B}\inter K & \leadsto &\hyper{H}_1^K,\ldots,\hyper{H}_p^K,
\end{array}$$
as well as (a unique) $\varphi:\set{1,\ldots,p}\rightarrow\set{1,\ldots,n}$, such that $H_k^K\inc H_{\varphi(k)}$ for all $1\leq k\leq p$. We observe that, by construction, $X_B\cap K=X_{B_1}\cap K$. Therefore, the third decomposition also induces (a unique) $\psi:\set{1,\ldots,p}\rightarrow\set{1,\ldots,m}$, such that $H_k^K\inc H^1_{\psi(k)}$ for all $1\leq k\leq p$. The left-hand side of ($P_1$) unfolds as follows:  
 $$\begin{array}{rcl}
\restrconstr{(\ast_B(\delta))}{K} & =&(X_ B(\ast(\delta\lceil_{H_1}),\dots,\ast(\delta\lceil_{H_n})))\lceil_{K} \\[0.1cm]
&=&(X_B\cap K)(\ast(\delta\lceil_{H_{\varphi(1)}})\lceil_{H^K_1}, \dots ,\ast(\delta\lceil_{H_{\varphi(p)}})\lceil_{H^K_p}) \\[0.1cm]
&=& (X_B\cap K) (\ast((\delta\lceil_{H_{\varphi(1)}})\lceil_{H^K_1}), \dots ,\ast((\delta\lceil_{H_{\varphi(p)}})\lceil_{H^K_p})) \\[0.1cm]
&=& (X_B\cap K)(\ast(\delta\lceil_{{H^K_1}}), \dots ,\ast(\delta\lceil_{{H^K_p}})), 
\end{array}$$
where we have successively used the definition of shuffle product, the definition of restriction, the induction hypothesis, and Remark \ref{rest1}. On the other hand, the right-hand side is computed in an analogous way, replacing $B$ by $B_1$, and $\varphi$ by $\psi$. Since $X_B\cap K=X_{B_1}\cap K$, we thus obtain $$\restrconstr{(\ast_B(\delta))}{K} =(X_{B_1}\cap K)(\ast(\delta\lceil_{{H^K_1}}), \dots ,\ast(\delta\lceil_{{H^K_p}}))=\restrconstr{(\ast_{B_1}(\delta))}{K}.$$
 
\smallskip
 
For the proof of  the property (P$_2$) , we consider only the decomposition $\hyper{H},X_B   \leadsto   \hyper{H}_1,\ldots,\hyper{H}_n$. We notice that the assumption $B_1=\emptyset$ amounts to $X_b\cap K=\emptyset$, which entails that there exists a unique $1\leq i\leq n$ such that $K\subseteq H_i$. The left-hand side now calculates as 
 $$\begin{array}{rcl}
\restrconstr{(\ast_B(\delta))}{K} & =&(X_ B(\ast(\delta\lceil_{H_1}),\dots,\ast(\delta\lceil_{H_n})))\lceil_{K} \\[0.1cm]
&=& (\ast (\delta\lceil_{H_i}))\lceil_K \\[0.1cm]
&=& \ast ((\delta\lceil_{H_i}) \lceil_K) \\[0.1cm]
&=& \ast (\delta\lceil_{K}),
\end{array}$$
where we followed exactly the same steps as above.  

\smallskip
 
Finally,
using the properties  (P$_1$) and  (P$_2$), we calculate
$$\begin{array}{l} \restrconstr{(\ast(\delta))}{K} \\
 =  \Sigma_{\emptyset\neq B\inc A} q^{|B|-1} \restrconstr{(\ast_B(\delta))}{K}\\
 =  \Sigma_{\emptyset\neq B_1\inc A_1,B_2\inc A_2} q^{|B_1|+|B_2|-1} \restrconstr{(\ast_B(\delta))}{K} +
\Sigma_{\emptyset\neq B_2\inc A_2} q^{|B_2|-1} \restrconstr{(\ast_B(\delta))}{K}\\
 =  \Sigma_{\emptyset\neq B_1\inc A_1} q^{|B_1|-1}(\Sigma_{B_2\inc A_2} q^{|B_2|})\restrconstr{(\ast_{B_1}(\delta))}{K} +
(\Sigma_{\emptyset\neq B_2\inc A_2} q^{|B_2|-1}) \ast(\restrconstr{\delta}{K}).
\end{array}$$
Now we observe (using $q=-1$) that 
$$\begin{array}{l}
\Sigma_{B_2\inc A_2} q^{|B_2|} = (1+q)^{|A_2|}=0\\
\Sigma_{\emptyset\neq B_2\inc A_2} q^{|B_2|-1} = \frac{(1+q)^{|A_2|}-1}{q}= 1,
\end{array}$$
and we conclude that $ \restrconstr{(\ast(\delta))}{K} = \Sigma_{\emptyset\neq B_1\inc A_1}0 \:\restrconstr{(\ast_{B_1}(\delta))}{K} +
\ast(\restrconstr{\delta}{K}) = \ast(\restrconstr{\delta}{K})$.
\end{proof}

We next embark on the proof that, for $q=-1$, we get  a polydendriform structure in the anti-strict  setting. 
We need one more definition. A  clan  $\Xi$ is called  \emph{associative} if,
for all $\tau  =  (\setc{\hyper{H}_a}{a\in A},\hyper{H})\in \Xi$, $a_0\in A$ and
$\tau'= (\setc{\hyper{H}_{(a_0,a')}}{a'\in A'},\hyper{H}_{a_0}) \in\Xi$, we have that
$$\tau'':=(\{{\bf H}_{a}\,|\,a\in A\backslash \{a_0\}\}\cup \{\hyper{H}_{(a_0,a')}\,|\,a'\in A'\},{\bf H})\in\Xi.$$  We  refer to $\tau''$ as the grafting of  
$\tau'$ to  $\tau$ along $a_0$. We can now state the main theorem of the article, illustrated in Figure \ref{figExple}.

\begin{theorem}\label{assoc-antistrict}
Let $\Xi$ be an associative anti-strict clan, and suppose that $\tau = (\setc{\hyper{H}_a}{a\in A},\hyper{H})\in \Xi$, $a_0\in A$, and $\tau'= (\setc{\hyper{H}_{(a_0,a')}}{a'\in A'},\hyper{H}_{a_0})\in\Xi$.
Suppose that we are given constructs $C_a:\hyper{H}_a$, for all $a\in A\backslash \{a_0\}$, 
and constructs $C_{(a_0,a')}:\hyper{H}_{(a_0,a')}$, for all $a'\in A'$. Taking $\tau''$ to be the grafting of $\tau'$ to  $\tau$ along $a_0$ and  setting $A'':=(A\backslash\{a_0\})\cup\{(a_0,a')\,|\,a'\in A'\}$, denote the corresponding delegations by $\delta''=(\tau'',\{C_{a''}\,|\,a''\in  A'' \})$ and
 $\delta'=(\tau',\{C_{(a_0,a')}\,|\,a'\in A' \})$.
Then, assuming that $q=-1$ in \eqref{shuffle-antistrict-def}, for each $\emptyset\neq B''\subseteq A''$, the following \emph{polydendriform equation} holds, 
$$
\ast^{\tau''}_{B''}(\delta'')=\begin{cases}
\ast^{\tau}_{B''}(\{C_a\,|\,a\in A\backslash\{a_0\}\}\cup\{\ast^{\tau'}(\delta')\}), & \mbox{ if } B''\subseteq A\backslash\{a_0\}\\
\ast^{\tau}_{B}(\{C_a\,|\,a\in A\backslash\{a_0\}\}\cup \{\ast^{\tau'}_{B'}(\delta')\}), & \mbox{ if } B''\not\inc A\backslash\{a_0\}
\end{cases}
,$$   where, in the second case, $B=(B''\cap(A\backslash\{a_0\}))\cup\{a_0\}$ and $B'=\{a'\in A'\,|\, (a_0,a')\in B''\}$. Moreover, the polydendriform equation  implies the following associativity equation: $$\ast^{\tau''}(\delta'')=\ast^{\tau}(\{C_a\,|\,a\in A\backslash\{a_0\}\}\cup\{\ast^{\tau'}(\delta')\}).$$ 
\end{theorem}

\begin{figure}
\centering
\resizebox{\textwidth}{!}{
\includegraphics[height=4cm]{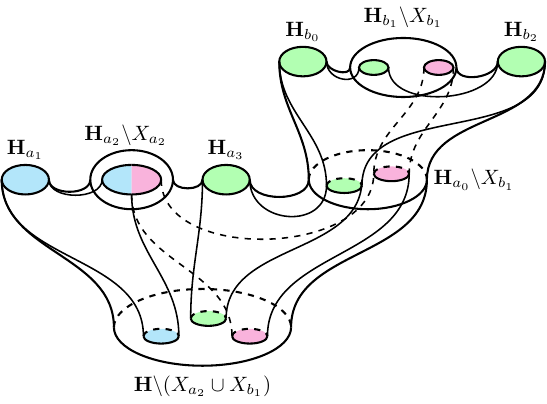}
\hspace{1cm}
\includegraphics[height=4cm]{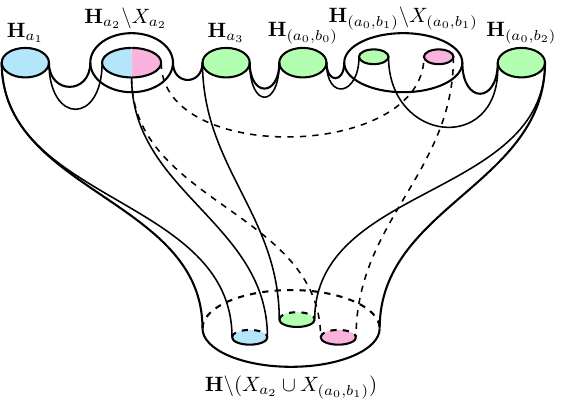} }
\caption{Illustration of 
Theorem \ref{assoc-antistrict}} \label{figExple}
\end{figure}

\begin{proof} We prove the polydendriform equation  and the associativity equation  together by induction on the size of ${\bf H}$.
Set $X_{B''}:=\Union_{b''\in B''}X_{b''}$, where $X_{b''}=\text{root}(C_{b''})$, and consider the decomposition 
$$\hyper{H},X_{B''}\leadsto H_1,\ldots,H_n.$$
By definition, we have that 
$$\ast^{\tau''}_{B''}(\delta'')=X_{B''}(\ast({\delta''}_{\lceil_{H_1}}),\dots,\ast({\delta''}_{\lceil_{H_n}})).$$
We distinguish two cases, corresponding to the two cases of the equation.
\begin{enumerate}
\item $B''\inc A\setminus\set{a_0}$. Write $\delta$ for the delegation $\{C_a\,|\,a\in A\backslash\{a_0\}\}\cup\{\ast^{\tau'}(\delta')\}$ of support $\tau$. In order to
 prove that $\ast^{\tau''}_{B''}(\delta'')=\ast^{\tau}_{B''}(\delta)$, since 
$$\ast^{\tau}_{B''}(\delta)=X_{B''}(\ast(\delta_{\lceil_{H_1}}),\dots,\ast(\delta_{\lceil_{H_n}})),$$
we have to show that for each $1\leq i\leq n$, $\ast({\delta''}_{\lceil_{H_i}})=\ast(\delta_{\lceil_{H_i}})$. For this, we first note that $\delta$ is obtained from $\delta''$ by replacing the set of  constructs $\{C_{(a_0,a')}:{\bf H}_{(a_0,a')}\,|\, a'\in A'\}$ of $\delta'$ with the single linear construct $\ast^{\tau'}(\delta')$ of ${\bf H}_{a_0}$. For $1\leq i\leq n$, there are two possible cases: either $H_i\cap H_{a_0}=\emptyset$, in which case we have, by definition, ${\delta''}_{\lceil_{H_i}}=\delta_{\lceil_{H_i}}$, and the conclusion follows obviously, or  $H_i\cap H_{a_0}\neq\emptyset$.  In the latter case, we unfold the definitions of both ${\delta''}_{\lceil_{H_i}}$ and $\delta_{\lceil_{H_i}}$ : $${\delta''}_{\lceil_{H_i}}=\{{(C_a)}_{\lceil_{H_i}}\,|\,a\in A\backslash\{a_0\},\, H_a\cap H_i\neq \emptyset\}\cup {\delta'}_{\lceil_{H_i}}$$
and $$\delta_{\lceil_{H_i}}=\{{(C_a)}_{\lceil_{H_i}}\,|\,a\in A\backslash\{a_0\},\, H_a\cap H_i\neq \emptyset\}\cup \{(\ast^{\tau'}(\delta'))_{\lceil_{H_i}}\}.$$
By Proposition \ref{shuffle-restriction}, we have that $(\ast^{\tau'}(\delta'))_{\lceil_{H_i}}=\ast^{\tau'}({\delta'}_{\lceil_{H_i}})$, which allows us to conclude by induction.

\item $B''\not\inc A\setminus\set{a_0}$. We now set 
$\delta=\{C_a\,|\,a\in A\backslash\{a_0\}\}\cup\{\ast^{\tau'}_{B'}(\delta')\}$. Our goal is to prove $\ast_{B''}(\delta'')=\ast_{B}(\delta)$. Let $X_{B'}:=\Union_{b'\in B'}X_{(a_0,b')}$.  Noticing that  
$$\Union_{b\in B}X_b=X_{B'}\union(\Union_{b\in B\setminus\set{a_0}}X_b)=X_{B''},$$ we have that
$$\ast^\tau_B(\delta)=X_{B''}(\ast(\delta_{\lceil_{H_1}}),\dots,\ast(\delta_{\lceil_{H_n}})).$$
Therefore, our goal reduces to proving that $\ast({\delta''}_{\lceil_{H_i}})=\ast(\delta_{\lceil_{H_i}})$, for each $1\leq i\leq n$. Once again, we note that $\delta$ is obtained from $\delta''$ by replacing the set of  constructs $\{C_{(a_0,a')}:{\bf H}_{(a_0,a')}\,|\, a'\in A'\}$ of $\delta'$ with the single linear construct $\ast_{B'}^{\tau'}(\delta')$ of ${\bf H}_{a_0}$.  As in the first part of the proof, we distinguish  two possible cases for $1\leq i\leq n$: either $H_i\cap H_{a_0}=\emptyset$, in which case we conclude as before,  or  $H_i\cap H_{a_0}\neq\emptyset$.  In the latter case, ${\delta''}_{\lceil_{H_i}}$ unfolds exactly as before, while for $\delta_{\lceil_{H_i}}$ we have 
$$\delta_{\lceil_{H_i}}=\{{(C_a)}_{\lceil_{H_i}}\,|\,a\in A\backslash\{a_0\},\, H_a\cap H_i\neq \emptyset\}\cup \{(\ast_{B'}^{\tau'}(\delta'))_{\lceil_{H_i}}\}.$$

 Let us define $B'_1=\{b'\in B'\,|\, X_{b'}\cap H_i\neq\emptyset\}$.  Then, since $H_i\cap X_{B''}=\emptyset$, and $X_{b'}\subseteq X_{B''}$ for each $b'\in B'$, we have that $B'_1=\emptyset$. Therefore, by the property (P$_2$)  of Proposition \ref{shuffle-restriction},  we have
$$(\ast_{B'}^{\tau'}(\delta'))_{\lceil_{H_i}}=\ast({\delta'}_{\lceil_{H_i}}),$$
which again allows us to conclude by induction. 
 \end{enumerate}
Once the polydendriform equation is proven, the associativity follows easily, exactly as in the proof of \cite{PLBJ1}[Theorem 4.17].
\end{proof}

We refer to  \cite[Proposition 4.18]{PLBJ1} for the  exact relation (in an ordered setting) between the present polydendriform structure and the notion of  tridendiform structure, originally introduced by Loday and Ronco under the name of dendriform trialgebra structure \cite{LR-tri}. Informally, ``polydendriform" is the varying-arity generalisation of ``tridendriform":  if $A=\set{l,r}$ has cardinal two, then there are three choices of $\emptyset\neq B\subseteq A$,
and the operations $\ast_{\set{l}}, \ast_{\set{l,r}}, \ast_{\set{r}}$  read as $\prec, \bcdot,\succ$ of the tridendriform literature.
 
 \subsection{Compatibility with the strict setting}\label{strict-and-antistrict-coincidence-section}
 Our next goal is to show that, for clans that are both strict and anti-strict, the shuffle product defined in Section \ref{shuffle-definition-section} coincides with the one defined in \cite{PLBJ1}.

\smallskip
 
The following technical  lemma will serve us to connect the two settings.
\begin{lemma}\label{oldway}
Let $\tau=(\{{\bf H}_a\,|\,a\in A\},{\bf H})$ be an anti-strict team. Define,  for each decomposition \eqref{decomp}, 
 $\tilde{A}:=(A\backslash B)\union\setc{(b,i)}{b\in B, 1\leq i\leq n_b}$. Then,
 for each $\tilde{a}\in \tilde{A}$  and  each $i\in\set{1,\ldots,n_B}$ such  $H_{\tilde{a}}\inter H_i\neq\emptyset$, the hypergraph $(\hyper{H}_{\tilde{a}})_{H_i}$ is connected.
\end{lemma}

\begin{proof}
Let $\tilde{a}\in \tilde{A}$ and $i\in\set{1,\ldots,n}$ be such that $H_{\tilde{a}}\inter H_i\neq\emptyset$.

If $\tilde{a}\in A\backslash B$, then, by instantiating Definition \ref{anti-strict-definition} for $K:=H_i$, we have that $H_{\tilde{a}}\cap H_i$ is connected in ${\bf H}_{\tilde{a}}$, hence the conclusion. 

If $\tilde{a}=(b,j)$  for some $b\in B$ and $1\leq j\leq n_b$, then we a fortiori have that $H_{b}\cap H_i\neq\emptyset$. By instantiating again Definition \ref{anti-strict-definition} for $K:=H_{i}$, we get that $H_{b}\cap H_i$ is connected in ${\bf H}_b$. By construction, we have that $H_i\cap X_b=\emptyset$, which implies that $H_{b}\cap H_i$ is contained in a unique ${\bf H}_{(b,k)}$, for some $1\leq k\leq n_b$. By the assumption that $H_{(b,j)}\cap H_i\neq \emptyset$, we conclude that $k=j$. Therefore, $H_{(b,j)} \cap H_i = H_b \cap H_i$, and hence the conclusion.
\end{proof}

We now recall from \cite{PLBJ1} the definition of strict clan and of the associated shuffle product.  Let $\tau= (\{{\bf H}_a\,|\, a\in A\},{\bf H})$ be a strict team, i.e.,  a preteam satisfying the following property:
\begin{itemize}
\item[(S)]  for every $a\in A$ and every non-empty connected subset $\emptyset\neq K\inc H_a$, we have that $K$  is also connected in $\hyper{H}$.  
\end{itemize}
An equivalent definition  is that, for each   subset $\emptyset\neq B\subseteq A$, we have,
in the notation of Lemma \ref{oldway}, that for all $\tilde{a}\in \tilde{A}$, $H_{\tilde{a}}$  is included in some $H_i$, thus inducing preteams 
$\tau_i^B=(\{{\bf H}_{\tilde{a}}\,|\, \tilde{a}\in \tilde{A} \:\mathrm{and} \: H_{\tilde{a}}\inc H_i^B\},{\bf H}_i^B) $, 
for all $1\leq i\leq n_B$. 
A strict clan is a set $\Xi$ of strict teams which is closed under decomposition, i.e.,  if  $\tau\in\Xi$, then all induced preteams $\tau_i^B$ are in $\Xi$ as well (for all $B$).
We note that, writing $C_b=X_b(C_{(b,1)},\ldots,C_{(b,n_b)})$, 
a delegation $\delta$ of support $\tau$  induces delegations
$\delta_i^B=(\{C_{\tilde{a}}:{\bf H}_{\tilde{a}}\,|\, \tilde{a}\in \tilde{A} \:\mathrm{and} \: H_{\tilde{a}}\inc H_i^B\},{\bf H}_i^B) $ of support $\tau_i^B$. The definition of the shuffle product given in \cite{PLBJ1} is 
\begin{equation} \label{shuffle-strict-def}
\ast(\delta)= \sum_{\emptyset\incs B\subseteq A} q^{{|B|}-1}\ast_B(\delta), \;\mathrm{where}
\ast_B(\delta)=X_B(\ast(\delta_1^B),\ldots,\ast(\delta_{n_B}^B)).
\end{equation}

\begin{remark}
In the strict framework, we were led to define the shuffle product in the unbiased manner (in the sense that the number of participating hypergraphs is not limited to 2) by the following observation.  Starting from $\tau$ as above with $|A|=2$, it could happen that an induced team $\tau_i^B$ has more than two participating hypergraphs. On the other hand, this cannot happen in the anti-strict setting: the induced teams \eqref{lll} are always such that the number of their participating hypergraphs is at most 2, which makes it possible to define the shuffle product as a binary operation and derive the $n$-ary versions by associativity. Nevertheless, we chose the unbiased route here as well, in order to obtain a uniform framework for the shuffle product.
\end{remark}

\begin{proposition} \label{strict-antistrict-agree}
In a clan that is both strict and anti-strict, the definitions  \eqref{shuffle-strict-def} and
\eqref{shuffle-antistrict-def} coincide.
\end{proposition}
\begin{proof}  We first prove that 
we can rephrase \eqref{shuffle-antistrict-def} as follows: if $\delta = (\{C_a:{\bf H}_a\,|\, a\in A\},{\bf H})$ is a delegation of support $\tau$, we have
$$\ast_B(\delta)=X_B(\ast(\varepsilon_1^B),\ldots,\ast(\varepsilon_{n_B}^B)),$$ where $$\varepsilon_i^B=(\setc{\restrconstr{(C_{\tilde a})}{H_{\tilde{a}}\inter H_i}}{H_{\tilde{a}}\inter H_i\neq\emptyset},\hyper{H}_i^B).$$ (By Lemma \ref{oldway}, the delegations $\varepsilon_i^B$ are well-defined.) This follows from two observations:
\begin{itemize}
\item if $H_b\cap H_i\neq\emptyset$, for $b\in B$, then there exists a unique $1\leq j\leq n_b$, such that $H_b\cap H_i=H_{(b,j)}\cap H_i$,
\item if ${C_b}_{\lceil_{H_b\cap H_i}}={C_{(b,j)}}_{\lceil_{H_{(b,j)}\cap H_i}}$,
\end{itemize} 
which entail that $\varepsilon_i^B$ as defined above is precisely $\delta\lceil_{H_i}$. 
On the other hand, the additional strictness assumption immediately gives $\delta_i^B=\varepsilon_i^B$, hence the conclusion.
\end{proof}

\begin{remark} \label{quasi-strict-and-antistrict}
Proposition
\ref{strict-antistrict-agree} actually also holds replacing ``strict'' by ``quasi-strict''.   A quasi-strict team is defined as a strict team, except that (with the notation above)
\begin{itemize}
\item[$(\dagger)$]  it is allowed for $H_{\tilde{a}}$ not to be included in some $H_i^B$, in which case it is required that each element of $H_{\tilde{a}}$ is  among the $H_i^B$'s. 
\end{itemize}  Then one replaces the indexing set $\tilde{A}$ by 
$\overline{A}:=(\tilde{A}\backslash \tilde{A}_d) + \union_{\tilde{a}\in \tilde{A}_d} H_{\tilde{a}}$, where
$\tilde{A}_d$ is the set of elements $\tilde{a}\in \tilde{A}$ such that  ($\dagger$) applies.  The construct associated with each element $\overline{x}$ of some $H_{\tilde{a}}$ with $\tilde{a}\in\tilde{A}_d$ is just $\overline{x}$.
And the definition of shuffle product is adapted, mutatis mutandis (we refer to \cite{PLBJ1} for details). The proof of Proposition  \ref{strict-antistrict-agree} is easily adapted by noticing that, for $\tilde{a}$ and $\overline{x}$ as above, with $\set{\overline{x}}=H_i^B$, we have trivially $\overline{x}=\restrconstr{(C_{\tilde{a}})}{H_{\tilde{a}}\cap H_i^B}$.
\end{remark}

\begin{remark} \label{not-subsume}
We note however that, in the strict case, the assumption $q=-1$ is not necessary (we needed
this assumption only for the quasi-strict case).
\end{remark}

We end the section by pointing out that most examples from \cite{PLBJ1} are in fact also anti-strict examples.
The verifications are left to the reader.

\begin{example} \label{strict-antistrict-examples} Our key examples of strict clans are
associahedra and permutohedra. Both of them are anti-strict.
\begin{itemize}
\item Associahedra. We take as universe the collection of all associahedra. The strict clan associated with associahedra is obtained by considering the set of all  preteams $$(\{\hyper{K}^{V_1},\ldots,\hyper{K}^{V_n}\}, \hyper{K}^V),$$ where $V_1,\ldots,V_n$ form successive intervals of $V$. The reader will notice that this is a special case of what we considered in Example \ref{treesassoc}: here, we concatenate linear graphs rather than insert them into one another   in some fancier way: this is the price to pay to maintain strictness (while preserving anti-strictness).
\item Permutohedra.  We take as universe the collection of all permutohedra. The strict clan associated with permutohedra is obtained by considering the set of all  preteams $$(\{\hyper{P}^{V_i}\}_{i \in I}, \hyper{P}^V),$$ where $\{V_i\}_{i \in I}$ forms a partition of $V$.
\end{itemize}
\end{example}

\begin{example} \label{quasistrict-antistrict-examples}
We group here our three examples of quasi-strict clans that are not strict, based on simplices, hypercubes, and erosohedra, respectively. Each of them is anti-strict.
\begin{itemize}
\item Simplices. The universe formed by all simplices $\hyper{S}^X$ (for a  finite set $X$) gives rise to the quasi-strict clan formed by all preteams of the form 
$$(\setc{\hyper{S}^{X_a}}{a\in A},\hyper{S}^{\Union X_a})$$ (for mutually disjoint $X_a$).  
\item Hypercubes.  The universe formed by all   hypercubes $\hyper{C}^X$ (for $X=\set{x_1<\cdots<x_n}$)  gives rise to the  quasi-strict clan formed by all preteams of the form 
$$(\{\hyper{C}^{X_1},\ldots,\hyper{C}^{X_n}\},\hyper{C}^{\Union X_i}),$$ where $\Union_{1\leq i\leq n} X_i$ is endowed with the  order in which  $X_1,\ldots,X_n$ form successive intervals. 
\item Erosohedra. The family of \emph{erosohedra} is given by:
\begin{align*}
 \hyper{E}^X=\set{\set{x_1},\ldots,\set{x_n},\set{x_1,\ldots,x_n}}\cup \setc{\setc{x_j}{j \neq i}}{1\leq i\leq n},
\end{align*} 
where $X=\{x_1, \ldots, x_n\}$. It gives rise to a quasi-strict clan containing all preteams of the following two forms:
$$(\set{\hyper{E}^{X_1},\ldots, \hyper{E}^{X_n}},\hyper{E}^{\sqcup X_i}) \mbox{ and }  
(\set{\hyper{H}^{X_1},\ldots, \hyper{H}^{X_n}},\hyper{S}^{\sqcup X_i}),$$ where all the $\hyper{H}^{X_i}$'s are erosohedra except one which is a simplex. 
\end{itemize} \end{example}

Here is the only example from \cite{PLBJ1} which is strict but not anti-strict.
\begin{example} \label{friezohedra-non-example}
 Consider the infinite graph $\Friezo$ on $\mathbb{Z}$ with the set of edges $\{(x,y) ||x-y|\leq 2 \}$, and take as universe the set of all its restrictions $\Friezo_X$ to finite sets $X=\{x_1<\dots<x_n\} \subseteq \mathbb{Z}$ such that $\Friezo_X$ is connected. We call these graphs \emph{friezohedra}. Note that $\Friezo_X$  is connected  exactly when there is no $i$ such that $x_{i+1}-x_i>2$. The strict clan associated with friezohedra is obtained by considering the set of all strict teams $$(\{\hyper{F}_{V_i}\}_{i \in I}, \hyper{F}_V),$$ where $\{V_i\}_{i \in I}$ forms a partition of $V$ and each of the hypergraphs $\hyper{F}_{V_i}$, as well as $\hyper{F}_{V}$, are connected.  
 We shall exhibit  a preteam in this strict clan that is not anti-strict. Take $V=\set{1,2,3,4,5}$,  $V_1=\set{1,2,4}$ and $V_2=\set{3,5}$, and consider $\tau=(\set{\hyper{F}_{V_1},\hyper{F}_{V_2}},\hyper{F}_V)$. Let $K=\set{1,3,4}$. We have that $K$ is connected in $\hyper{F}_{V}$, but $K\cap V_1=\set{1,4}$ is not connected in $\hyper{F}_{V_1}$.
\end{example}

We end with another example of a strict clan which is not anti-strict, not present in \cite{PLBJ1}.
\begin{example}[Strict clan which is not anti-strict]\label{strict-not-antistrict}
Consider the preteam $\tau=(({\bf S}^{\{1,2,3\}},{\bf S}^{\set{4}}),{\bf K}^{\{1,2,3,4\}})$. Recall from Section \ref{fundamental-examples-section} that the  participating hypergraphs ${\bf S}^{\{1,2,3\}}$ and ${\bf S}^{\set{4}}$ are simplices  (the first one on three vertices and  the other one on a single vertex) and that the coordinating hypergraph ${\bf K}^{\{1,2,3,4\}}$ is the associahedron on four vertices.   It is easy to verify that $\tau$ has the  property (S). On the other hand, the connected subset $K=\{1,2\}$ of  ${\bf K}^{\{1,2,3,4\}}$ gets disconnected in $K\cap {S}^{\{1,2,3\}}=K$, which is a counter-example to $\tau$ being anti-strict.

This toy example readily expands to an example of a strict clan which is not anti-strict: we consider the universe consisting of simplices and associahedra, and we  take $\Xi$  to be the set of   preteams $$(({\bf S}^{X_1},\dots,{\bf S}^{X_n}),{\bf K}^{\bigcup_{1\leq i\leq n} X_i}),$$ for some $n\geq 1$, where $X_1<X_2<\cdots < X_n$. To see that $\Xi$ is strict, notice that each $X_i$ is connected in ${\bf K}^{\bigcup_{1\leq i\leq n} X_i}$. On the other hand, $\Xi$ is not anti-strict as it contains the team $\tau$ from above.
\end{example}

\begin{remark}\label{restrictohedra-remark}
In  \cite{PLBJ1}, we have introduced universes obtained by taking all possible connected restrictions $\hyper{H}_X$ of a fixed ``big'' hypergraph $\hyper{H}$. We called them \emph{restrictohedra} and showed that they all give rise to strict clans. They include associahedra, permutohedra, and friezohedra.  It then follows from Example \ref{friezohedra-non-example} that
not all clans coming from restrictohedra are anti-strict.
\end{remark}

\begin{proposition} 
Strictness and  anti-strictness are incomparable notions, i.e. there exist universes with clans that are anti-strict and not strict, and there are universes with clans that are strict but not anti-strict.
\end{proposition}
\begin{proof}
By Examples \ref{treesassoc} and \ref{strict-not-antistrict}.
\end{proof}

It can also be checked that all the examples of anti-strict  clans given in the present paper are associative.

\subsection{An operadic formulation} \label{operadic-section}
In this section (and only in this section), we assume the reader to have a basic familiarity with the notion of operad and of algebra over an operad (see e.g. \cite{MARKL} for a short introduction, and \cite{MSS} or \cite{LV} for book-length accounts). We define the polydendriform operad ${\bf Polydend}_C$ as a linear, coloured, non-skeletal, non-unital and componential operad, which is parameterised  by a class $C$ of colours, and we recast both the strict and anti-strict settings in terms of algebras for this operad (instantiated by taking the collection of colours to be a collection of  finite hypergraphs). We briefly recall the meaning of the terms used above to qualify the flavour of the operad that we are about to  define:

\begin{itemize}
\item linear means that the operation spaces are vector spaces (here over $\mathbb{R}$);
\item coloured means ``multisorted" or ``multicategorical" : each operation has a codomain and several domains, associated with each of its inputs;
\item non-skeletal means that the sets of operations are not indexed over natural numbers (recording the number of inputs) but by finite sets (giving names to the input positions) (see \cite{Obradovic17,MSS});
\item non-unital means that there is no unit nor unit laws, i.e., there are only compositions of operations;
\item componential (also called partial, or ``circle $i$" ) means that operations are composed in pairs, one receiving the other on one of its inputs (the one named $i$)\footnote{It is known (see \cite{MARKL}) that in the absence of units, the componential definition is not equivalent to the other (more standard) definition of composition, where one operation receives operations on all of its inputs at once.}.
\end{itemize}

\smallskip
In such an operad ${\cal O}$ , one is given the following data.
\begin{enumerate}
\item  for each finite set $I$  (a set of input positions) and  each {\em type} $(\set{c_i}_{i\in I},c)$, consisting of  an indexed family assigning  colours to all input positions, and an output colour:
a vector space ${\cal O}(\set{c_i}_{i\in I},c)$;
\item an action of bijections $\sigma:I\rightarrow J$ on those sets;
\item and for each $I$, $i_0\in I$ and $J$ such that $I\setminus\set{i_0}$ and $J$ are disjoint, and for each triple made of an operation $f\in {\cal O}(\set{c_i}_{i\in I},c)$,  a choice of $i_0\in I$, and an operation
$g\in {\cal O}(\set{c'_j}_{j\in J},c_{i_0})$: an operation $f\circ_{i_0}  g\in {\cal O}(\set{c''_k}_{k\in (I\setminus\set{i_0}\cup J},c)$, where $c''_k=c_k$ if $k\in I$ and
$c''_k=c'_K$ if $k\in J$, subject to axioms of associativity and equivariance.
\end{enumerate}

The operad ${\bf Polydend}_C$ is the operad presented by the following generators and relations.
\begin{itemize}
\item
For each type  $\alpha = (\set{c_i}_{i\in I},c)$, we have generators $\ast^\alpha$ and
$\ast_{I'}^\alpha$ where $I'$ ranges all  subsets of $I$.  
We abbreviate $\ast_\emptyset^\alpha$ as $\ast$.
\item 
We ask for the following relations to hold, where $\alpha = (\set{c_i}_{i\in I},c)$, 
$\beta=(\set{c'_j}_{j\in J},c_{i_0})$, and $\gamma=(\set{c''_k}_{k\in K},c)$  are as above:
%
$$\begin{array}{ll}
\ast^\alpha =  \sum_{\emptyset\neq B \subseteq I} q^{|B|-1}\:\ast_{B}^\alpha & \\\\
\ast^\gamma_{K'} = \ast^\alpha_{(K'\cap I)\cup\set{i_0}}\circ_{i_0}\ast^\beta_{K'\cap J} & \textrm{for all}\; \emptyset \neq K'\inc K.
\end{array}$$
\end{itemize}

We now take $C$ to be a certain collection $\mathfrak{U}$ of finite hypergraphs -- what we called above a universe. 
Then types are what we called preteams on this universe. Let us fix an anti-strict clan $\Xi$ on $\mathfrak{U}$  
(cf. Definition \ref{antistrict-clan-definition}). 
For a preteam $\tau=(\{{\bf H}_a\,|\, a\in A\},{\bf H})$,
\begin{itemize}
\item if $\tau\in \Xi$, we interpret the generators $\ast^\tau_{\emptyset}$ and $\ast^\tau_B$ (for $B$ a non-empty subset of $A$) as the functions $\ast$ and $\ast_B$ defined in 
Section \ref{shuffle-definition-section};
\item  if $\tau\not\in \Xi$, we interpret  all
the corresponding generators as the constant $0$ function.
\end{itemize}
Then (up to some obvious notational changes) Theorem \ref{assoc-antistrict} amounts to affirm that the above defines an algebra on ${\bf Polydend}_{\mathfrak{U}}$\footnote{Strictly speaking, Theorem \ref{assoc-antistrict} takes care only of a half of the associativity laws -- the so-called sequential one, but the rest of the laws, i.e., parallel associativity (in which two operations are plugged on two inputs) and equivariance, is straightforward to check.}.

\smallskip
The same holds for ``strict'' in the place of ``anti-strict''. 

\section{Comparison with other works} \label{comparison-section}
In this section  we discuss the works of  Ronco \cite{RoncoGTO}, Chapoton  \cite{Cha1,Cha2} and Forcey-Springfield \cite{FS}, and draw some comparisons with our work.  
\subsection{The framework of Maria Ronco  \cite{RoncoGTO}}
 In \cite{RoncoGTO},  Mar\' ia Ronco defines a  product $\ast$   for certain families of graph associahedra. More precisely,  she  gives  a condition,  that we shall refer to as (R1) below,  that allows her to state her definition and characterise it in terms of what she calls a generalised Tamari order.
 Moreover, she proves the associativity of $\ast$ under a strengthened condition, below referred to as (R2).

\smallskip

We shall now reformulate the construction of Ronco in the language of preteams.  In our setting,  Ronco's    condition  (R1)  translates as the following requirement on preteams. We shall say that a  preteam $(\{{\bf H}_a\,|\,a\in A\},{\bf H})$ is an (R1){\em -team} if it satisfies the following property:
\begin{itemize}
\item[(R1)] for every $a\in A$ and every subset $\emptyset\neq K\subseteq H_a$, we have that,  if  $K$ is connected in ${\bf H}$, then it is   connected in ${\bf H}_a$. 
\end{itemize}
 The strengthened  condition (R2) of Ronco additionally requires strictness for (R1)-teams, i.e.,  (R2) =  (R1) $\cap$ (S) (see \S \ref{strict-and-antistrict-coincidence-section}).  

\smallskip 

As for the relation between Ronco's conditions and the concepts introduced in this paper (and in its predecessor \cite{PLBJ1}), we note the following facts.
  \begin{itemize}
\item By definition, (AS) implies (R1).
  \item By Example \ref{stellohedra-example} below,  (R1) does not imply (AS).
  \item By Example \ref{teleass} below,  (S) does not imply (R1).
\end{itemize}
In addition, under the (R2) condition, by Proposition 4.24 of \cite{PLBJ1}, Ronco's product coincides with our product, for $q=1$. Hence our strict framework is more general than (R2), even when restricting ourselves to graph associahedra and to taking $q=1$. Another difference is that we additionally provide a polydendriform decomposition of the product, which is absent in Ronco's approach. 

We sum up there relations on the Figure \ref{schemaMaria}.

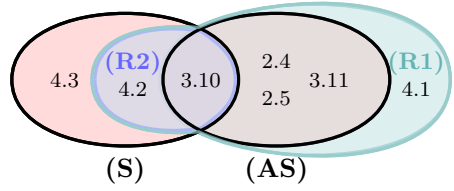
\begin{figure}[H]
\centering
 \begin{tikzpicture}
\draw[very thick,teal!40,fill=teal!20,fill opacity=0.4] (1.956,0) [out=90,in=200,looseness=0.86] to (2.5,0.715) [out=25,in=90]to (5.8,0);
\draw[very thick,teal!40,fill=teal!20,fill opacity=0.4] (1.956,0) [out=-90,in=160,looseness=0.86] to (2.5,-0.715) [out=-25,in=-90]to (5.8,0);
\draw[very thick,fill=red!20,fill opacity=0.7]   (0,0) [out=90,in=90] to (3,0)  [out=-90,in=-90] to (0,0);
\draw[very thick,fill=red!20,fill opacity=0.7] (2,0) [out=90,in=90] to (5,0)  [out=-90,in=-90] to (2,0);
\draw[very thick,blue!40,fill=blue!20,fill opacity=0.4](2.956,0)[out=90,in=-35] to (2.5,0.64) [out=170,in=90] to(1.1,0);
\draw[very thick,blue!40,fill=blue!20,fill opacity=0.4](2.956,0)[out=-90,in=35] to (2.5,-0.64) [out=190,in=-90] to(1.1,0);
\draw[very thick,teal!40,fill=teal!20,fill opacity=0.4](2.475,0.69)[out=175,in=90]to(1.06,0) to[out=-90,in=185] (2.475,-0.69) [out=-25,in=-90,looseness=0.9]to (5.8,0) to[out=90,in=25] (2.47,0.69);
\draw[very thick]   (0,0) [out=90,in=90] to (3,0)  [out=-90,in=-90] to (0,0);
\draw[very thick] (2,0) [out=90,in=90] to (5,0)  [out=-90,in=-90] to (2,0);
\draw[very thick,teal!40,opacity=0.7](2.475,0.69)[out=175,in=90]to(1.06,0) to[out=-90,in=185] (2.475,-0.69) [out=-25,in=-90,looseness=0.9]to (5.8,0) to[out=90,in=25] (2.47,0.69);
  \node at (2.5,0) {\footnotesize \ref{strict-antistrict-examples}};
\node at (4.2,0) {\footnotesize \ref{quasistrict-antistrict-examples}};
 \node at (3.5,0.25) {\footnotesize \ref{treesassoc}};\node at (5.35,-0.12) {\footnotesize \ref{stellohedra-example}}; 
\node at (3.5,-0.25) {\footnotesize \ref{excyc}}; 
\node at (0.7,0) {\footnotesize \ref{teleass}}; \node at (1.6,-0.12) {\footnotesize \ref{fgh}};
  \node at (1.5,-1.2) {\bf (S)};   \node at (3.5,-1.2) {\bf (AS)};
  \node at (5.35,0.22) {\small\color{teal!60}\bf  (R1)};   \node at (1.6,0.22) {\small\color{blue!60}\bf (R2)};
\end{tikzpicture}
\caption{Relations between the conditions (S), (AS), (R1) and (R2) on preteams. The numbers appearing in the different regions refer to examples in the text that illustrate the corresponding (combination of) properties. The associativity property holds only in the red part of the diagram, for any $q$ in (S) and for $q=-1$ in (AS). }
\label{schemaMaria}
\end{figure}
 
\begin{example}[The universe of stellohedra]  \label{stellohedra-example}
Let $(X,x)$ be a pointed set, i.e., a non-empty set $X$ with a distinguished element $x\in X$. The $(|X|-1)$-dimensional  \emph{stellohedron} is the polytope for the following graph:
$$\hyper{St}^{X,x}=\set{\setc{x}{x\in X}} \cup\setc{\set{x,y}}{y\in X\setminus\set{x}}.$$
We note that a non-singleton  subset $Y$ of $St^{X,x}$ is connected if and only if $x \in Y$.
Consider a preteam of the form 
$\tau=(\setc{\hyper{St}^{X_a,x_a}}{a\in A},\hyper{St}^{\sqcup_{a\in A} X_a, x_b})$ 
where $b$ is some chosen element of $A$ and all the $X_a$'s are disjoint.  Then $\tau$ is an (R1)-team.  Indeed, let, for some $a\in A$, $K\inc H_a$  be such that  $|K|>1$ and $K$ is connected in $\hyper{St}^{\sqcup_{a\in A} X_a, x_b}$. 
Then we have $b\in K$, which in turn entails that  $a=b$  and that $K$ is connected in $\hyper{St}^{X_a,x_a}$.

On the other hand,  $\tau$ is not anti-strict in general. Consider any $a\neq b$  and any $G\inc X_a\setminus{\{x_a\}}$, Then $K:=G\sqcup\set{b}$ is connected in $\hyper{St}^{\sqcup_{a\in A} X_a, x_b}$ but $K\cap St^{X_a,x_a}=G$ is not connected in $\hyper{St}^{X_a,x_a}$, as soon as $\hyper{St}^{X_a,x_a}$ has cardinal $>2$.
\end{example}

\begin{example}\label{fgh}
Referring to Example \ref{friezohedra-non-example}, showing that friezohedra form a strict clan which is not anti-strict,  we  also observe that    the corresponding preteams are (R1)-teams in the sense of Ronco. Therefore, they are (R2), and hence the class (R2) is broader than (S) $\cap$ (AS).
\end{example}

\begin{example}[The universe of teleassociahedra]\label{teleass}
Let $V\inc\mathbb{Z}$ be a non-empty finite set of integers, and let $d$ be a positive integer. We call $(V,d)$-\emph{teleassociahedron} the polytope for the hypergraph
\begin{equation}
    \hyper{T}^d_V:=\{\{i\}\,|\, i \in V\} \cup \{\{i,j\}\,|\, i,j \in V, i < j \mbox{ and } |i-j| \leq d\},
\end{equation}
when this hypergraph is connected.  

We see that two vertices of $\hyper{T}^d_V$  are connected by an edge whenever their distance is ``within the connectivity radius $d$'', so that, as $d$ grows (on a fixed vertex set $V$), the graph becomes denser (see Figure \ref{tas}). In particular,  for $V=\set{1,\ldots,n}$ and $d=1$, we have that $\hyper{T}^1_V={\bf K}^V$, while, for arbitrary $V$ and an integer $d$ such that
$d\geq \max(V)-\min(V)$,  we get that $\hyper{T}^d_V={\bf P}^V$. Therefore, we recover associahedra and permutohedra as particular instances of teleassociahedra. 

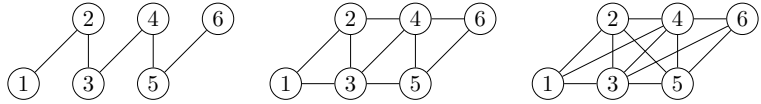
\begin{figure}[H]
\centering
\resizebox{3cm}{!}{\begin{tikzpicture}
\node[v] (1) at (0,0) {1};
\node[v] (2) at (1,1) {2};
\node[v] (3) at (1,0) {3};
\node[v] (4) at (2,1) {4};
\node[v] (5) at (2,0) {5};
\node[v] (6) at (3,1) {6};
\draw (6)--(5)--(4)--(3)--(2)--(1);
\end{tikzpicture}} \quad \resizebox{3cm}{!}{\begin{tikzpicture}
\node[v] (1) at (0,0) {1};
\node[v] (2) at (1,1) {2};
\node[v] (3) at (1,0) {3};
\node[v] (4) at (2,1) {4};
\node[v] (5) at (2,0) {5};
\node[v] (6) at (3,1) {6};
\draw (6)--(5)--(4)--(3)--(2)--(1);
\draw (1)--(3)--(5);
\draw(2)--(4)--(6);
\end{tikzpicture}} \quad \resizebox{3cm}{!}{\begin{tikzpicture}
\node[v] (1) at (0,0) {1};
\node[v] (2) at (1,1) {2};
\node[v] (3) at (1,0) {3};
\node[v] (4) at (2,1) {4};
\node[v] (5) at (2,0) {5};
\node[v] (6) at (3,1) {6};
\draw (6)--(5)--(4)--(3)--(2)--(1);
\draw (1)--(3)--(5);
\draw (2)--(4)--(6);
\draw (1)--(4);
\draw (2)--(5);
\draw (3)--(6);
\end{tikzpicture}}
\caption{From left to right: teleassociahedra $\hyper{T}^1_{V}$, $\hyper{T}^2_{V}$ and $\hyper{T}^3_{V}$, for $V=\set{1,2,3,4,5,6}$.}\label{tas}
\end{figure}

 In order to state a connectedness criterion for $\hyper{T}^d_V$, we first note that every finite subset $\emptyset \neq V\subseteq \mathbb{Z}$ can be written uniquely as a disjoint union  $V=\sqcup_{i=1}^n V_i$, for some $n \geq 1$, such that  each $V_i$  is an interval of $\mathbb{Z}$,
 and such that $\min(V_i)-\max(V_{i-1})>1$, for each $1 < i \leq n$; we shall denote this presentation of $V$ by writing $V=V_1 <\ldots < V_n$. Then the hypergraph  $\hyper{T}^d_V$ is connected if and only if $d\geq\max(\min(V_i)-\max(V_{i-1}))$, for each $1 < i \leq n$.     

\smallskip

The universe formed by all  teleassociahedra gives rise to a strict clan $\Xi_T$, formed by all preteams of the form 

$$\left(\left( \hyper{T}^{d_1}_{V_1}, \ldots, \hyper{T}^{d_n}_{V_n} \right), \hyper{T}^d_V\right),$$ where $V= V_1 <\ldots < V_n$ and $d \geq \max_{1\leq i \leq n} d_i$.
Indeed, consider the following observations:
\begin{itemize}
\item[1)] If $W\inc V$, we have $(\hyper{T}^d_V)_W=\hyper{T}^d_W$.
\item[2)] For any $X\inc V$, we can write $\hyper{T}^d_V, X\leadsto T_{W_1}^d,\ldots,T_{W_k}^d$ in such a way that
$V\backslash X=W_1 < \cdots < W_k$.  To see this, suppose, say, that $i,j,x,y,z$ are such that  $x,y\in W_i$, $z\in W_j$ and
$x<z<y$. Then by connectedness we have $(y-x)\leq d$, but then a fortiori $(y-z)\leq d$, which implies that
$T_{W_i\cup\set{z}}^d$ is connected, contradicting the maximality of $W_i$ as a connected component of $\hyper{T}^d_V\backslash X$.
\item[3)] If $d\leq e$, $V\inc W$ and $\hyper{T}^d_V$ is connected, then $V$ is connected in $\hyper{T}^e_W$.
\end{itemize}
 The property (S) for a preteam as above is proved as follows. Fix an index $1\leq i\leq n$ and a non-empty connected subset $\emptyset \neq K\subseteq T^{d_i}_{V_i}$. By observation 1), we have that $(\hyper{T}^{d_i}_{V_i})_K=\hyper{T}^{d_i}_K$. By the connectedness assumption for $K$, we know  that $\hyper{T}^{d_i}_K$ is connected. 
 Then, by observation 3), since $d_i\leq d$, $K\subseteq V$ and ${\bf T}^{d_i}_{K}$ is connected, we have that $K$ is connected in $\hyper{T}^d_V$. It remains to show that
$\Xi_T$ is closed under decomposition, as required by the definition of a strict clan (recalled in \S\ref{quasi-strict-and-antistrict}). This is a direct consequence of  observation 2), which shows that preteams induced by decompositions of the form \eqref{decomp} have the same shape as above.

\smallskip

 The clan $\Xi_T$ does not fit in Ronco's framework, as it fails to satisfy the condition (R1). For instance, consider the strict team $\left( \left( \hyper{T}_{\set{1,2,3}}^1, \hyper{T}^1_{\{4\}} \right), \hyper{T}_{\set{1,2,3,4}}^2 \right)$. The edge $\{1,3\}$ belongs to $\hyper{T}^2_{\set{1,2,3,4}}$, but not to  $\hyper{T}_{\set{1,2,3}}^1$.

\end{example}
 \subsection{The framework of Fr\' ed\' eric Chapoton \cite{Cha1}}
In \cite{Cha1}, Chapoton defines  Hopf algebras on the faces of permutohedra, associahedra and hypercubes. 
While the algebra part of these Hopf algebras coincide with our product on permutohedra and associahedra, this is not the case for hypercubes, as we show below. 
 He uses the well-known description of the $(n-1)$-dimensional permutohedron  in terms of surjections $[n]\doublearrow{}[d]$,  for $n\geq 1$ and $1\leq d\leq n$. Given such a surjection $f$, he defines a word $w(f)$ of 
length $n-1$ over the alphabet $\{+,-,\bullet\}$ as follows: $$w(f)(i)=\begin{cases}
+ , \mbox{ if } f(i+1)>f(i) \\
\bullet , \mbox{ if } f(i+1)=f(i) \\
- , \mbox{ if } f(i+1)<f(i) 
\end{cases}.$$
The correspondence $w$ extends to a surjection from the poset of faces of the $(n-1)$-dimensional permutohedron to the poset of faces of the $(n-1)$-dimensional hypercube. We give $w$ explicitly for $n=3$  in Figure \ref{formap}, where we also translate Chapoton's words (blue) to our words (red).

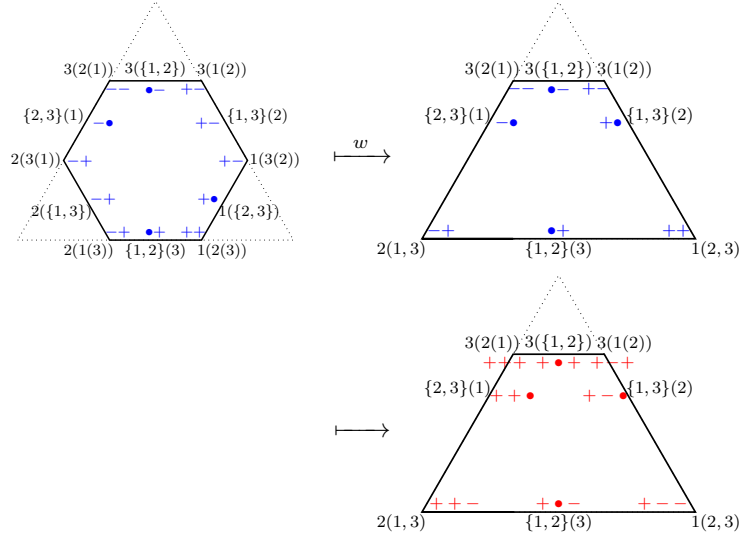
\begin{figure}[H]
\quad\quad\quad\begin{tabular}{rcl}
\resizebox{4cm}{!}{\begin{tikzpicture}[scale=5]
\def\s{1}
\def\t{0.333333}
\coordinate (A) at (0,0);
\coordinate (B) at (\s,0);
\coordinate (C) at ({\s/2},{0.8660254*\s});
\coordinate (A1) at ($(A)!\t!(B)$);
\coordinate (A2) at ($(A)!\t!(C)$);

\coordinate (B1) at ($(B)!\t!(C)$);
\coordinate (B2) at ($(B)!\t!(A)$);

\coordinate (C1) at ($(C)!\t!(A)$);
\coordinate (C2) at ($(C)!\t!(B)$);
\node (a1) at ($(A1)+(-0.075,-0.04)$)  {\footnotesize $2(1(3))$};
\node (a2) at ($(A2)+(-0.1,0)$)  {\footnotesize $2(3(1))$};
\node (b2) at ($(B2)+(0.075,-0.04)$)  {\footnotesize $1(2(3))$};
\node (b1) at ($(B1)+(0.1,0)$)  {\footnotesize $1(3(2))$};
\node (c1) at ($(C1)+(-0.075,0.04)$)  {\footnotesize $3(2(1))$};
\node (c2) at ($(C2)+(0.075,0.04)$)  {\footnotesize $3(1(2))$};
\node (ab) at (0.5,-0.04)  {\footnotesize $\{1,2\}(3)$};
\node (ac) at (0.13,0.45) {\footnotesize $\{2,3\}(1)$};
\node (bc) at (0.87,0.45) {\footnotesize $\{1,3\}(2)$};
\node (d) at (0.5,0.62) {\footnotesize $3(\{1,2\})$};
\node (e) at (0.16,0.1) {\footnotesize $2(\{1,3\})$};
\node (f) at (0.84,0.1) {\footnotesize $1(\{2,3\})$};

\node (a1') at ($(A1)+(0.03,0.03)$)  {\color{blue} \footnotesize $-+$};\node (a1') at (0.31,0.15)  {\color{blue} \footnotesize $-+$};\node (a1') at (0.225,0.285)  {\color{blue} \footnotesize $-+$};
\node (b2') at ($(B2)+(-0.03,0.03)$)  {\color{blue} \footnotesize $++$};\node (b2'') at (0.69,0.15)  {\color{blue} \footnotesize $+\bullet$};\node (b2''') at (0.775,0.285)  {\color{blue} \footnotesize $+-$};
\node (ab') at (0.5,0.03)  {\color{blue}\footnotesize $\bullet +$};
\node (c1') at ($(C1)+(+0.03,-0.03)$)  {\color{blue} \footnotesize $--$};\node (c1'') at (0.31,0.425)  {\color{blue} \footnotesize $-\bullet$};\node (c1''') at (0.69,0.425)  {\color{blue} \footnotesize $+-$};\node (c1'''') at ($(C2)+(-0.03,-0.03)$)  {\color{blue} \footnotesize $+-$};
\node (ac') at (0.5,0.545)  {\color{blue}\footnotesize $\bullet -$};
\draw[dotted]
  (A) -- (B) -- (C) -- cycle;

\draw[thick]
  (A1) -- (B2) -- (B1) -- (C2) -- (C1) -- (A2) -- cycle;

\end{tikzpicture}} \!\!\!&\!\raisebox{1.35cm}{$\xmapsto{\enspace w\enspace }$}\!&\!\!\!\!\!\!\!\!\!\!\!\!\resizebox{5cm}{!}{\begin{tikzpicture}[scale=4.5]

\def\s{1}

\def\t{0.333333}

\coordinate (A) at (0,0);
\coordinate (B) at (\s,0);
\coordinate (C) at ({\s/2},{0.8660254*\s});

\coordinate (A1) at ($(A)!\t!(B)$);
\coordinate (A2) at ($(A)!\t!(C)$);

\coordinate (B1) at ($(B)!\t!(C)$);
\coordinate (B2) at ($(B)!\t!(A)$);

\coordinate (C1) at ($(C)!\t!(A)$);
\coordinate (C2) at ($(C)!\t!(B)$);
\node (a1) at ($(A)+(-0.075,-0.04)$)  {\footnotesize $2(1,3)$};
\node (b2) at ($(B)+(0.075,-0.04)$)  {\footnotesize $1(2,3)$};
\node (c1) at ($(C1)+(-0.075,0.04)$)  {\footnotesize $3(2(1))$};
\node (c2) at ($(C2)+(0.075,0.04)$)  {\footnotesize $3(1(2))$};
\node (ab) at (0.5,-0.04)  {\footnotesize $\{1,2\}(3)$};
\node (ac) at (0.13,0.45) {\footnotesize $\{2,3\}(1)$};
\node (bc) at (0.87,0.45) {\footnotesize $\{1,3\}(2)$};
\node (d) at (0.5,0.62) {\footnotesize $3(\{1,2\})$};

\node (a1') at ($(A)+(0.07,0.03)$)  {\color{blue} \footnotesize $-+$};
\node (b2') at ($(B)+(-0.07,0.03)$)  {\color{blue} \footnotesize $++$}; 
\node (ab') at (0.5,0.03)  {\color{blue}\footnotesize $\bullet +$};
\node (c1') at ($(C1)+(+0.03,-0.03)$)  {\color{blue} \footnotesize $--$};\node (c1'') at (0.31,0.425)  {\color{blue} \footnotesize $-\bullet$};\node (c1''') at (0.69,0.425)  {\color{blue} \footnotesize $+\bullet$};\node (c1'''') at ($(C2)+(-0.03,-0.03)$)  {\color{blue} \footnotesize $+-$};
\node (ac') at (0.5,0.545)  {\color{blue}\footnotesize $\bullet -$};
\draw[dotted]
  (A) -- (B) -- (C) -- cycle;

\draw[thick]
  (A1) --(A)-- (B) -- (C2) -- (C1) -- (A) -- cycle;

\end{tikzpicture}} \\
\!\!\!&
\raisebox{1.35cm}{$\xmapsto{\quad\enspace }$}&\!\!\!\!\!\!\!\!\!\!\!\!\resizebox{5cm}{!}{\begin{tikzpicture}[scale=4.5]

\def\s{1}

\def\t{0.333333}

\coordinate (A) at (0,0);
\coordinate (B) at (\s,0);
\coordinate (C) at ({\s/2},{0.8660254*\s});

\coordinate (A1) at ($(A)!\t!(B)$);
\coordinate (A2) at ($(A)!\t!(C)$);

\coordinate (B1) at ($(B)!\t!(C)$);
\coordinate (B2) at ($(B)!\t!(A)$);

\coordinate (C1) at ($(C)!\t!(A)$);
\coordinate (C2) at ($(C)!\t!(B)$);
\node (a1) at ($(A)+(-0.075,-0.04)$)  {\footnotesize $2(1,3)$};
\node (b2) at ($(B)+(0.075,-0.04)$)  {\footnotesize $1(2,3)$};
\node (c1) at ($(C1)+(-0.075,0.04)$)  {\footnotesize $3(2(1))$};
\node (c2) at ($(C2)+(0.075,0.04)$)  {\footnotesize $3(1(2))$};
\node (ab) at (0.5,-0.04)  {\footnotesize $\{1,2\}(3)$};
\node (ac) at (0.13,0.45) {\footnotesize $\{2,3\}(1)$};
\node (bc) at (0.87,0.45) {\footnotesize $\{1,3\}(2)$};
\node (d) at (0.5,0.62) {\footnotesize $3(\{1,2\})$};

\node (a1') at ($(A)+(0.12,0.03)$)  {\color{red} \footnotesize $++-$};
\node (b2') at ($(B)+(-0.12,0.03)$)  {\color{red} \footnotesize $+--$}; 
\node (ab') at (0.5,0.03)  {\color{red}\footnotesize $+\bullet -$};
\node (c1') at ($(C1)+(-0.03,-0.03)$)  {\color{red} \footnotesize $+\!+\!+$};\node (c1'') at (0.33,0.425)  {\color{red} \footnotesize $++\bullet$};\node (c1''') at (0.67,0.425)  {\color{red} \footnotesize $+-\bullet$};\node (c1'''') at ($(C2)+(+0.03,-0.03)$)  {\color{red} \footnotesize $+\!-\!+$};
\node (ac') at (0.5,0.545)  {\color{red}\footnotesize $+\bullet +$};
\draw[dotted]
  (A) -- (B) -- (C) -- cycle;

\draw[thick]
  (A1) --(A)-- (B) -- (C2) -- (C1) -- (A) -- cycle;

\end{tikzpicture}}
\end{tabular}
\caption{Interestingly, this shows that the quotient turning the hexagon into a square does not correspond to reversing the truncations (that produce the hexagon from the simplex).  For example, the constructs $1(3(2))$, $3(1(2))$ and $\{1,3\}(2)$ are identified in the quotient, yet they do not arise from a common vertex truncation of the permutohedron.}
\label{formap}
\end{figure}
 
The example given in Figure \ref{formap} is a particular instance of the following lemma.
\begin{lemma}\label{words}
An order-preserving bijection between the constructs of the hypercube and the words of Chapoton is obtained by mapping the constructs to the words of \cite[Example 3.4]{PLBJ1}, followed by deleting the first $+$ of each word, and switching the role of $+$ and $-$ in each word. 
\end{lemma}
The correspondence of Lemma \ref{words} allows us to compare  the product of Chapoton with our product: 
$$\begin{array}{rcl}
2(1)\ast^{\text{Chap}} 3(4)={\color{blue}{-}}\ast^{\text{Chap}}{\color{blue}{+}} & = &({\color{blue}{-++}})+({\color{blue}{-\bullet+}})+({\color{blue}{--+}})\\
& = & ({\color{red}{++--}})+({\color{red}{++\bullet-}})+({\color{red}{+++-}})\\
&=&2(1,3,4)+\{2,3\}(1,4)+3(2(1),4)\end{array},$$
while 
$$\begin{array}{rcl}
2(1)\ast 3(4)&=&2(1,3,4)-\{2,3\}(1,4)+3(2(1),4),\end{array}$$
where $\ast$ is our product. Therefore, the two products are different. One might wonder whether our product, for the specific case of hypercubes,  also works for $q=1$ (and then try to compare the  two products), but this is not the case: for $q=1$, one can easily check that  $(1\ast 2)\ast 3 \neq 1\ast (2\ast 3)$.

Similar differences appear between our product on cyclohedra and on simplices with respect to those  obtained by Stefan Forcey and Derriell Springfield in \cite{FS} by means of suitable projections.

\medskip
So it seems that the approaches of Chapoton and Forcey-Springfields on the one hand and ours on the other hand have different ``virtues". In our case, we try to have a general framework  that we can instantiate ``individually'' to various  families of polytopes, but we know little about how they relate. In their case, most of their new structures are obtained by projecting a known one from one family to another family, which in fact turns the projection into an algebra morphism.

 \section*{Appendix: A non-inductive definition of restriction}

We give an alternative, non-inductive characterisation of the restriction of constructs. This characterisation is not used elsewhere in the paper, but it   may be of independent interest. We first introduce some preliminary notations.

\smallskip

Given $\hyper{H}$, a construct $S:\hyper{H}$ and $a,b\in H$,  there are unique nodes $X^S_a$ and $X^S_b$ of $S$ such that $a\in X^S_a$ and $b\in X^S_b$, respectively. We write
\begin{itemize}
\item $P_<(a,b,S)$ if  $X^S_a$ is strictly below  $X^S_b$ in $S$,
\item $P_=(a,b,S)$  if $X^S_a=X^S_b$,
\item $P_{<^\#}(a,b,S)$ if  $X^S_a$ is strictly below  $X^S_b$, or if $X^S_a$ and $X^S_b$ are incomparable in $S$, meaning that neither node lies above the other in $S$,
\item $P_{=^\#}(a,b,S)$ if $X^S_a=X^S_b$, or if $X^S_a$ and $X^S_b$ are incomparable in $S$.
\end{itemize}
We note that if $T$ is a subtree of $S$ and if $X^S_a$ and $X^S_b$ are nodes of $T$, then $P_{\ast}(a,b,T)$ holds if and only if 
$P_{\ast}(a,b,S)$ holds, for any $\ast\in\set{<,=,<^\#,=^\#}$.

\begin{proposition} \label{characterisation-restriction}
For a connected hypergraph $\hyper{H}$, a construct $U:\hyper{H}$,  a connected subset $K$ of $\hyper{H}$ and a construct $S:\hyper{K}$ (with $\hyper{K}:={\bf H}_K$),
we have that $S=\restrconstr{U}{K}$ if and only  the following properties hold for all $a,b\in K$:
$$\begin{array}{lll}
(1)\quad P_<(a,b,S) \Rightarrow P_<(a,b,U) && (2)\quad  P_<(a,b,U) \Rightarrow P_{<^\#}(a,b,S) \\
(3) \quad P_=(a,b,S) \Rightarrow P_=(a,b,U)  && (4)\quad  P_=(a,b,U) \Rightarrow P_{=^\#}(a,b,S).
\end{array}$$
\end{proposition}
\begin{proof} 
We first prove that $\restrconstr{U}{K}$ satisfies the properties (1)--(4), for all $a,b\in K$, by induction on $U$. If $U=H$, then $\restrconstr{U}{K}=K$, and, hence, the implications (1) and (3) are vacuously true (since the antecendants are false), while (2) and (4)   hold trivially given the single-node form of $\restrconstr{U}{K}$ and $U$. Suppose now that ${\bf H},X\leadsto {\bf H}_1,\dots,{\bf H}_n$ and that  $U=X(U_1,\dots,U_n)$, where $U_j:{\bf H}_j$, for each $1\leq j\leq n$. If $X\cap K=\emptyset$, then $\restrconstr{U}{K}=\restrconstr{(U_j)}{K}$ for some $1\leq j\leq n$ and the claim holds by induction on $U_j$ (relatively to $K$).  It remains to prove the claim for the case
$X\cap K\neq \emptyset$. In this case, supposing that ${\bf K},(X\inter K)\leadsto K_1,\dots,K_p$, the restriction has the form 
\begin{equation}
\restrconstr{U}{K}=(X\cap K)((U_{\psi(1)})_{\lceil_{K_1}},\dots,(U_{\psi(p)})_{\lceil_{K_p}})\tag{$\dagger$}\label{UUU}\end{equation}
 with $\psi$ being the appropriate index correspondence.   

We first show (1). Supposing that $P_<(a,b,\restrconstr{U}{K})$, we consider the following two possibilities. 
\begin{itemize}
\item If $a\in X\cap K=\text{root}(\restrconstr{U}{K})$, then $a\in X=\text{root}(U)$. By the assumption, $b$ appears in a node of $\restrconstr{U}{K}$ which is strictly above $\text{root}(\restrconstr{U}{K})$. So, there exists an index $1\leq i\leq p$, such that $b$ appears in a node of $(U_{\psi(i)})_{\lceil_{K_i}}$, meaning that $b\in K_i$. But ${H}_{\psi(i)}\cap X=\emptyset$, implying that ${K}_{i}\cap X=\emptyset$, and hence   $b\not\in X=\text{root}(U)$. Therefore,  $P_<(a,b,U)$ holds.
\item Otherwise, there exists an index $1\leq i\leq p$, such that $a,b\in K_i$. In this case, we conclude by applying induction to 
$U_{\psi(i)}$ (with respect to the connected subset $K_i$ of ${\bf H}_{\psi(i)}$).
\end{itemize}
In order to show (2), suppose that $a,b\in K$ are such that $P_<(a,b,U)$ holds. Once again, we distinguish  two  cases. 
\begin{itemize}
\item If $a\in X=\text{root}(U)$, then $a\in X\cap K=\text{root}(\restrconstr{U}{K})$. By the assumption, $b$ appears in a node of $U$ which is strictly above $\text{root}(U)$, so $b\not\in X$, and consequently $b\notin \text{root}(\restrconstr{U}{K})$, meaning that  $P_<(a,b,\restrconstr{U}{K})$ holds.
\item Otherwise, there exists an index $1\leq j\leq n$ such that $a,b\in H_j$. In this case, there are two possibilities:
\begin{itemize}
\item there exist $i_1\neq i_2\in (\psi)^{-1}(j)$ such that $a\in K_{i_1}$ and $b\in K_{i_2}$; in this case, $a$ and $b$ lie in incomparable nodes of 
$\restrconstr{U}{K}$;
\item there exists $i\in (\psi)^{-1}(j)$ such $a,b\in K_i$; in this case we conclude by induction to $U_j$ (with respect to $K_i$).
\end{itemize}

\end{itemize}
The proofs of (3) and (4) are similar.

\smallskip
In the other direction, we show that if the properties (1)--(4) hold for a construct $S:{\bf K}$, then $S=\restrconstr{U}{K}$.
We proceed again by induction on $U$. $U=H$, then $\restrconstr{U}{K}=K$. Since $S$ satisfies (4) and $P_=(a,b,U)$ holds for all $a,b\in K$, $P_{=^\#}(a,b,S)$ must hold for all $a,b\in K$. Suppose that for some $a,b\in K$, $X^S_a$ and $X^S_b$ are incomparable in $S$, and let $Y$ be the first common ancestor of $X^S_a$ and $X^S_b$ in $S$. For each $y\in Y$, $P_{<}(y,a,S)$ holds, so, since $S$ satisfies (1), $P_<(y,a,U)$ must hold as well. This is in contradiction with the single-node form of $U$, so we conclude that for each $a,b\in K$,  $X^S_a=X^S_b$, i.e., that $S$ is the single-node construct $K=\restrconstr{U}{K}$. Suppose now that ${\bf H},X\leadsto {\bf H}_1,\dots,{\bf H}_n$ and that  $U=X(U_1,\dots,U_n)$, where $U_j:{\bf H}_j$, for each $1\leq j\leq n$.  As in the proof of (1), if $X\cap K=\emptyset$, then $\restrconstr{U}{K}=\restrconstr{(U_j)}{K}$ for some 
$1\leq j\leq n$ and we can apply induction to $U_j$ (with respect to $K$ and $S$). If $X\cap K\neq\emptyset$, then $\text{root}(\restrconstr{U}{K})=X\cap K$. Suppose that $\text{root}(S)=Y$. We  show that  $X\cap K=Y$.  Suppose that there exists an  $a\in K$, such that $a\in Y$ but $a\not\in X\cap K$, i.e., 
$a\not\in X$, and let $b\in X\cap K$. Then $P_<(b,a,U)$, and hence $P_{<^\#}(b,a,S)$ by our assumption, but this is impossible since $b\in \text{root}(S)$.
Suppose now that $a$ witnesses the failure of the other inclusion, i.e.,
$a\in X\cap K$ but $a\not\in Y$. Then, taking $b\in Y$, we have $P_<(b,a,S)$ and hence $P_<(b,a,U)$, contradicting the fact that $a\in \text{root}(\restrconstr{U}{K})$.  Therefore $S$ has the form $(X\cap K)(S_1,\ldots, S_p)$. 
Using the notation of \eqref{UUU},
we can then apply induction 
to each $U_{\psi(i)}$ (with respect to $K_i$ and $S_i$) and conclude 
that  $S_i=\restrconstr{(U_{\psi(i)})}{{K}_i}$ for all $1\leq i\leq p$, i.e., that $S=\restrconstr{U}{K}$. \end{proof}

\printbibliography
 
\end{document}